\documentclass[11pt]{amsart}
\usepackage[utf8]{inputenc}
\usepackage{amsthm}
\usepackage{amsmath}
\usepackage{amsfonts}
\usepackage{latexsym}
\usepackage[margin=1in]{geometry}
\usepackage{enumerate}
\usepackage{csquotes}
\usepackage{graphicx}
\usepackage{tikz-cd}
\usepackage{comment}
\usepackage{appendix}
\usepackage{hyperref}
\usepackage{bm}
\usepackage{amssymb}
\emergencystretch=\maxdimen
\DeclareMathOperator{\Rm}{Rm}
\DeclareMathOperator{\Ric}{Ric}
\DeclareMathOperator{\sing}{sing} 
\DeclareMathOperator{\Rc}{Rc}

\makeatletter
\newcommand*{\rom}[1]{\rm {\expandafter\@slowromancap\romannumeral #1@}}
\makeatother

\usepackage[style=alphabetic,doi=false,isbn=false,url=false,eprint=false]{biblatex}
\numberwithin{equation}{section}
\newtheorem{Theorem}{Theorem}[section]
\newtheorem{Proposition}[Theorem]{Proposition}
\newtheorem{Lemma}[Theorem]{Lemma}

\newtheorem{maintheorem}{Theorem}

\theoremstyle{definition}

\newtheorem{Remark}[Theorem]{Remark}
\newtheorem{Defn}[Theorem]{Definition}

\def \IP {\mathbb{P}}

\def \ii {\sqrt -1}
\usepackage{xcolor}

\title{Finite time singularities of the Ricci flow on compact K\"ahler surfaces are of Type I}
\author{Charles Cifarelli}
\address{Department of Mathematics, Tufts University, 177 College Avenue, Medford, MA 02155}
\email{Charles.Cifarelli@tufts.edu }

\author{Ronan J.~Conlon}
\address{Department of Mathematical Sciences, The University of Texas at Dallas, Richardson, TX 75080}
\email{ronan.conlon@utdallas.edu}

\author{Max Hallgren}
\address{
Department of Mathematics, The University of Arizona, 617 N. Santa Rita Ave., Tucson, AZ 85721}
\email{mhallgren@arizona.edu}

\author{Junsheng Zhang}
\address{Courant Institute of Mathematical Sciences, New York University, 251 Mercer St, New York,
NY 10012}
\email{jz7561@nyu.edu}

\date{}

\begin{document}

\begin{abstract}
    For any volume-collapsing finite-time singularity of a K\"ahler-Ricci flow on a compact K\"ahler surface, we show the flow satisfies a Type I curvature bound and classify the corresponding tangent flows. Combined with previous results, this shows that any finite-time singularity of a K\"ahler-Ricci flow on a compact K\"ahler surface is of Type I.
\end{abstract}

\maketitle

\setcounter{tocdepth}{1}

\section{Introduction}

\subsection{Background}

Suppose $(M^2,\widetilde{J},(\widetilde{g}_t)_{t\in [0,T)})$ is a solution of the K\"ahler-Ricci flow
\begin{equation*}
    \partial_t \widetilde{g}_t = -\Ric(\widetilde{g}_t)
\end{equation*}
on a compact K\"ahler surface. If the flow develops a singularity at time $T<\infty$, then there are three possible qualitative behaviors of the flow:
\begin{enumerate}
    \item (Noncollapsing) $\lim_{t \nearrow T}\operatorname{vol}_{\widetilde{g}_t}(M)>0$,

    \item (Collapsing) $\lim_{t \nearrow T} \operatorname{vol}_{\widetilde{g}_t}(M)=0$, $\limsup_{t \nearrow T} \operatorname{diam}_{\widetilde{g}_t}(M)>0$,

    \item (Extinction) $\limsup_{t\nearrow T} \operatorname{diam}_{\widetilde{g}_t}(M)=0.$
\end{enumerate}
By \cite{songnotes,TZ}, there is a surjective holomorphic map $\widetilde{\pi}:M\to Y$ with connected fibers onto a K\"ahler manifold $Y$ such that the limiting K\"ahler class $[\widetilde{\omega}_T]=[\widetilde{\omega}_0]-Tc_1(M)$ is equal to $\widetilde{\pi}^{\ast}[\omega_Y]$ for some K\"ahler metric $\omega_Y$ on $Y$. 

The extinction case, where $Y$ is a point, was the first of the above to be fully understood. In this case, it was shown in \cite{SongExt,TZ} that $M$ is a Fano manifold and $[\omega_0]=Tc_1(M)$. By \cite{tiancalabi,Koiso,WangZhu,SesumTian,TianZhu07,ChenWangsurfaces}, the flow satisfies the Type I curvature bound 
\begin{equation} \label{TypeIDef} \sup_M|\operatorname{Rm}_{\widetilde{g}_t}|_{\widetilde{g}_t} \leq \frac{C}{T-t} \qquad \text{ for all } t\in [0,T)
\end{equation}
for some constant $C>0$, and the rescalings $(M,(T-t)^{-1}\widetilde{\omega}_t)$ converge in the smooth Cheeger-Gromov sense as $t\nearrow T$ to the unique (up to biholomorphism, see \cite{BandoMabuchi,TianZhu02}) shrinking K\"ahler-Ricci soliton on $M$. 

The noncollapsing case corresponds to $\dim_{\mathbb{C}}(Y)=2$, where $\widetilde{\pi}$ is the blowdown map along finitely many disjoint $(-1)$-curves in $M$. It was shown in \cite{CHM} that the Type I bound \eqref{TypeIDef} also holds in this situation. Moreover, for any basepoint $p\in M$ on a blown-down $(-1)$-curve, \cite{CCD} states that the rescaled metrics $(M,(T-t)^{-1}\widetilde{\omega}_t,p)$ converge in the smooth pointed Cheeger-Gromov sense to the unique (up to biholomorphism) shrinking K\"ahler-Ricci soliton, constructed in \cite{FIK}, on the total space of $\mathcal{O}_{\mathbb{P}^1}(-1)$. By \cite{SW1}, there is moreover a canonical K\"ahler-Ricci flow $(Y,(\widehat{g}_t)_{t\in [T,T')})$ on $Y$ satisfying $\widetilde{\pi}^{\ast}\widehat{g}_T=\lim_{t \nearrow T}\widetilde{g}_t$  in the sense of currents, such that the flow through singularities is continuous in the Gromov-Hausdorff sense. In \cite{CHL}, it was shown that $(\widehat{g}_t)_{t\in (T,T')}$ satisfies an analogous Type I curvature bound flowing out of the singularity, and the rescalings $(Y,(t-T)^{-1}\widehat{g}_t,\widetilde{\pi}(p))$ converge in the pointed Cheeger-Gromov sense to a K\"ahler-Ricci expander on $\mathbb{C}^2$ constructed in \cite{Caosoliton}. 

The remaining case is therefore that of collapsing. In the special case of a ruled surface with $U(2)$ symmetry, it was shown in \cite{SWHirz,Fong} that the flow subsequentially converges in the Gromov-Hausdorff sense to a multiple of the Fubini-Study metric on $\mathbb{P}^1$, that the flow satisfies the Type I curvature bound (\ref{TypeIDef}), and that the tangent flow based at any point is the standard metric on $\mathbb{C}\times \mathbb{P}^1$. Without the symmetry assumption, it is known from \cite{SSW} that whenever $\widetilde{\pi}:M\to Y$ is a $\mathbb{P}^1$-bundle, the metrics $(M,d_{\widetilde{g}_t})$ subsequentially converge as $t\nearrow T$ to a metric space homeomorphic to $Y$. Furthermore, it was recently shown in \cite{XuZhang,JScollapsing} that under this same assumption, \eqref{TypeIDef} holds, and for any $p\in M$, the rescalings $(M,(T-t)^{-1}\widetilde{g}_t,p)$ converge in the pointed Cheeger-Gromov sense to the cylindrical metric on $\mathbb{C}\times \mathbb{P}^1$ (see also \cite{Licollapsing}). In general, however, $\widetilde{\pi}$ is a fibration with generic fiber $\mathbb{P}^1$, rather than a $\mathbb{P}^1$-bundle. By \cite{XuZhang}, a Type I curvature bound holds near any regular fiber of $\widetilde{\pi}$, so it remains only to understand the behavior near a singular fiber. Our main result in this paper is that the flow actually satisfies \eqref{TypeIDef} globally, and that the unique singularity model based at any singular fiber is the K\"ahler-Ricci shrinker on the blowup of $\mathbb{C}\times \mathbb{P}^1$ constructed in \cite{BCCD}.

\subsection{Statements of main results}

We suppose that $(M^2,\widetilde{J},(\widetilde{g}_t)_{t\in [0,T)})$ is a K\"ahler-Ricci flow on a compact K\"ahler surface which develops a singularity at time $T<\infty$. The collapsing case is equivalent to the following volume behavior at the singular time:
\begin{equation} \label{eq:volassumption} 0<\lim_{t\nearrow T} (T-t)^{-1}{\rm Vol}_{\widetilde{g}_t}(M) <\infty.\end{equation}
By \cite[Theorem 1.4]{TZ}, $M$ then admits the structure of a Fano fibration $\widetilde{\pi}:M \to Y$ over a smooth curve $Y$, and there is a K\"ahler metric $\omega_Y$ on $Y$ such that
\begin{equation} \label{eq:fibrationclass}
    [\widetilde{\omega}_T]=[\widetilde{\omega}_0]-Tc_1(M)=\widetilde{\pi}^{\ast}[\omega_Y].
\end{equation}
Because $\widetilde{\pi}$ is proper, the set $U\subseteq Y$ of regular values of $\widetilde{\pi}$ has finite complement by Remmert's proper mapping theorem, and the restriction of $\widetilde{\pi}$ to $\widetilde{U}:=\widetilde{\pi}^{-1}(U)$ is a holomorphic $\mathbb{P}^1$-bundle.

We may now state our main theorem, which gives a description of the small-scale behavior of $M$ at any point at the first singular time in terms of the map $\widetilde{\pi}$.  

\begin{maintheorem} \label{thm:main} 
A collapsing Ricci flow on a compact K\"ahler surface develops a Type I singularity. More precisely, in the situation described above, we have:
\begin{enumerate}[(i)]
     \item \label{thm:main:TypeI} The Type I curvature bound \eqref{TypeIDef} holds.

    \item \label{thm:main:regular} For any $p\in \widetilde{U}$ and any sequence $t_i \nearrow T$, there exists $p_{\infty} \in \mathbb{C}\times \mathbb{P}^1$ such that, after passing to a subsequence, the following convergence holds in the smooth pointed Cheeger-Gromov sense:
    \begin{equation*}
        (M,(T-t_i)^{-1}\widetilde{g}_{t_i},p)\to (\mathbb{C}\times \mathbb{P}^1,g_{\mathbb{C}}+g_{FS},p_{\infty}) \qquad \text{ as } i \to \infty,
    \end{equation*}
    where $g_{FS}$ is the Fubini-Study metric on $\mathbb{P}^1$.
    \item \label{thm:main:singular} For any $p\in M\setminus \widetilde{U}$ and any sequence $t_i \nearrow T$, there exists $p_{\infty} \in M_{\operatorname{BCCD}}$ such that, after passing to a subsequence, 
    the following convergence holds in the smooth pointed Cheeger-Gromov sense:
    \begin{equation*}
        (M,(T-t_i)^{-1}\widetilde{g}_{t_i},p)\to (M_{\operatorname{BCCD}},g_{\operatorname{BCCD}},p_{\infty}) \qquad \text{ as } i\to \infty,
    \end{equation*}
    where $(M_{\operatorname{BCCD}},g_{\operatorname{BCCD}})$ is the K\"ahler-Ricci shrinker constructed in \cite{BCCD}.
\end{enumerate}
\end{maintheorem}
\begin{Remark}
    In \cite{Bam3}, a metric space $(M_T,d_T)$ was defined which roughly serves as the time slice of the flow $(M,(\widetilde{g}_t)_{t\in [0,T)})$ at the singular time $t=T$. In \cite{CHL}, it was shown that $\widetilde{\pi}:M\to Y$ naturally induces a continuous surjective map $\widehat{\pi}:M_T\to Y$. In the course of proving Theorem \ref{thm:main}\eqref{thm:main:regular},\eqref{thm:main:singular}, we show that any tangent flow (in the sense of \cite{Bam3}) of $(M,(\widetilde{g}_t)_{t\in [0,T)})$ based at a point $\mu \in M_T$ is the shrinking cylinder on $\mathbb{C}\times \mathbb{P}^1$ if $\widehat{\pi}(\mu)\in U$, and the BCCD shrinker otherwise.  
\end{Remark}

\subsection{Outline of the paper}

In \S \ref{section:preliminaries}, we review notation and definitions, and state the compactness theorem for K\"ahler-Ricci flows which will be used in the remainder of the paper. In \S \ref{section:regular}, we give an independent proof of the fact that tangent flows near regular fibers of $\widetilde{\pi}$ are $\mathbb{C}\times \mathbb{P}^1$, which has recently been established by other methods in \cite{XuZhang,JScollapsing}. 

In \S \ref{section:singular}, we address the main difficulty in the proof of Theorem \ref{thm:main}, which is showing that any tangent flow of $(M,(\widetilde{g}_t)_{t\in [0,T)})$ at a singular fiber is in fact smooth. This is accomplished by the following steps.
\begin{itemize}
    \item[\S \ref{subsection:topology}] An elementary argument using the adjunction formula identifies a neighborhood of a singular fiber with the BCCD shrinker topologically.

    \item[\S \ref{subsection:fanofibration}] The argument of \cite{SunZhang} shows that $X$ is the total space of a polarized Fano fibration $\pi:X\to \mathcal{C}$ over an affine variety $\mathcal{C}$, hence $X$ has only finitely many orbifold singularities. We also show $X$ is noncompact and is not asymptotic to a K\"ahler cone, which is used to show $\mathcal{C}\cong \mathbb{C}$ and $X$ is smooth away from the central fiber of $\pi$.

    \item[\S \ref{subsection:centralfiber}] The central fiber of $\pi$ has at most two irreducible components. 

    \item[\S \ref{subsection:relativeMM}] Letting $\widetilde{X}$ denote the minimal resolution of $X$, the relative minimal model of $\widetilde{X}\to \mathbb{C}$ is then $\mathbb{P}^1\times \mathbb{C}$. 

    \item[\S \ref{subsection:Wahl}] We show that any singularity of $X$ must be a Wahl singularity, using a technical refinement of \cite[Theorem 5.10(ii)]{CHM}.

    \item[\S \ref{subsection:tangentflowisBCCD}] We show that $X$ has no singularities, by using the previously established facts to compute the trace of the intersection matrix of the central fiber of $\widetilde{X}\to X\to \mathbb{C}$ in two different ways. 
\end{itemize}
At this point, we have classified all possible tangent flows of $(M,(\widetilde{g}_t)_{t\in [0,T)})$. Using this, we complete the proof of Theorem \ref{thm:main} in \S \ref{section:TypeI}, which relies on the following contradiction-compactness argument. Assuming that, by way of contradiction, there is a sequence $(x_i,t_i)\in M\times [0,T)$ with $(T-t_i)|\operatorname{Rm}_{\widetilde{g}}|_{\widetilde{g}}(x_i,t_i)\to \infty$, we can find (after changing basepoints and passing to a subsequence) a region of the flow which is close to a nontrivial orbifold K\"ahler-Ricci shrinker $X_{\infty}$ at scales smaller than the Type I scale. A straightforward argument shows that such $X_{\infty}$ must be singular, so a further change of basepoints gives $(x_i',t_i')$ such that 
\begin{equation*}
\limsup_{i\to \infty}\mathcal{N}_{x_i',t_i'}(\tau)\leq \log \frac{1}{2}
\end{equation*}
for any fixed $\tau \in (0,T)$. 
On the other hand, we may pass to a subsequence so that $\mathcal{N}_{x_i',t_i'}(\tau)\to \mathcal{N}_{\mu}(\tau)$ for some $\mu \in M_T$. Because any tangent flow at $\mu$ has Nash entropy equal to that of the BCCD shrinker or $\mathbb{C}\times \mathbb{P}^1$, and both entropies are strictly larger than $\log(\frac{1}{2})$, we obtain 
\begin{equation*}
    \lim_{\tau \searrow 0} \mathcal{N}_{\mu}(\tau)>\log \frac{1}{2},
\end{equation*}
yielding a contradiction.

\subsection{Acknowledgments} CC is supported by NSF Grant DMS-2506521.
RJC is supported by a Simons Travel Grant.
MH is grateful to Jian Song for helpful discussions.

\subsection{AI disclosure} Most of the mathematical argument in Subsection \ref{subsection:Wahl} was developed by ChatGPT-6 Astra. ChatGPT was also used for literature searches, for a technical refinement of the argument of Proposition \ref{prop:components}, and to check for typos and small mistakes throughout. No part of the manuscript was written by AI.

\section{Preliminaries and notation}
\label{section:preliminaries}

Throughout the remainder of the paper, we fix a solution $(M^2,\widetilde{J},(\widetilde{g}_t)_{t\in [0,T)})$ of the K\"ahler-Ricci flow on a compact K\"ahler surface, and assume that \eqref{eq:volassumption} holds. We let $\widetilde{\pi}:M\to Y$ be as in \eqref{eq:fibrationclass}. Recall that $U \subseteq Y$ denotes the set of regular values of $\widetilde{\pi}$, and $\widetilde{U}=\widetilde{\pi}^{-1}(U)$. 

For any K\"ahler-Ricci flow $(M',(g_t')_{t\in I})$ of complex dimension $2$, defined on a time interval $I\subseteq \mathbb{R}$, we use the following notation. Given $(x,t)\in M'\times I$, we let $K^{g'}(x,t;\cdot,\cdot):M'\times (I\cap (-\infty,t)) \to (0,\infty)$ denote the conjugate heat kernel based at $(x,t)$, and define probability measures $d\nu^{g'}_{x,t;s}:=K^{g'}(x,t;\cdot,s)dg_s'$, where $dg'_s$ denotes the Riemannian volume measure of $(M',g'_s)$. We also use the notation
$$B_{g'}(x,t,r):= \{y\in M'; d_{g'_t}(x,y)<r\}$$
for the geodesic ball with respect to the Riemannian distance $d_{g'_t}$ of $(M',g'_t)$. The notion of a $P^{\ast}$-parabolic neighborhood was introduced in \cite{Bam1}: given $(x_0,t_0)\in M' \times I$, $P^{\ast}(x_0,t_0;A,-T^-,T^+)$ is the set of $(x,t) \in M' \times (I\cap [t_0 - T^-,t_0+T^+])$ satisfying
$$d_{W_1}^{g_{t_0-T^-}'}(\nu_{x_0,t_0;t_0-T^-}^{g'},\nu_{x,t;t_0-T^-}^{g'})<A.$$

For $(x,t)\in M'\times I$, we define the curvature scale
$$r_{{\Rm}}^{g'}(x,t):=\sup\{r>0;|{\Rm}_{g'}|_{g'}\leq r^{-2}\text{ on }B_{g'}(x,t,r)\times([t-r^{2},t+r^{2}]\cap I)\}.$$

For ease of notation, we also let $\widetilde{K},\widetilde{\nu}_{x,t;s},d\widetilde{g}_t$ denote the above quantities when $g'$ is our distinguished flow $(\widetilde{g}_t)_{t\in [0,T)}$. 

\begin{Defn} (See Section 2.6 of \cite{Bam3}) \label{conjsingtime} 
Let $M_T$ denote the set of conjugate heat flows $\mu=(\mu_t)_{t\in [0,T)}$ satisfying
\begin{equation*}
    \lim_{t\nearrow T} \int_M\int_M d_{\widetilde{g}_t}^2(x,y)d\mu_t(x)d\mu_t(y) =0.
\end{equation*}
$M_T$ is equipped with the metric $d_T$ given by
\begin{equation*}
    d_T(\mu^1,\mu^2):= \lim_{t\nearrow T} d_{W_1}^{\widetilde{g}_t}(\mu_t^1,\mu_t^2)
\end{equation*}
for $\mu^i=(\mu_t^i)_{t\in [0,T)}\in M_T$, where $d_{W_1}^{\widetilde{g}_t}$ denotes the 1-Wasserstein distance with respect to the metric $d_{\widetilde{g}_t}$. For $\mu=(\mu_t)_{t\in [0,T)}\in M_T$, write $d\mu_t = (2\pi(T-t))^{-2}e^{-f_t}d\widetilde{g}_t$, and define the Nash entropy based at $\mu$ by
\begin{equation*}
    \mathcal{N}_{\mu}(\tau):= \int_M f_{T-\tau} d\mu_{T-\tau} - 2.
\end{equation*}
We say $(x,t)$ is an $H_4$-center of $\mu$ if $\int_M d_{\widetilde{g}_t}^2(x,y)d\mu_t(y)\leq H_4(T-t)$, where $H_4:= \frac{3}{2}\pi^2+4$. 
\end{Defn}

We refer the reader to \cite[Chapter 4]{BG} for the definition of a K\"ahler orbifold and an overview of its basic properties. Given any orbifold $X$, we let $X_{\operatorname{reg}} \subseteq X$ denote the (dense, open) set of points with trivial isotropy group, so that $(X_{\operatorname{reg}},g)$ is a smooth K\"ahler manifold. 

\begin{Defn} An orbifold K\"ahler-Ricci shrinker $(X,g,J,f)$ consists of a smooth K\"ahler orbifold $(X,g,J)$ along with a smooth function $f\in C^{\infty}(X)$ satisfying 
$${\Rc}+\nabla^2f=g, \hspace{6 mm} \mathcal{L}_{\nabla f}J=0.$$
We say $X$ has isolated singularities if $X \setminus X_{\text{reg}}$ is discrete and the isotropy group of any point is a finite subgroup of $U(2)$ which acts freely on $\mathbb{S}^3$.
\end{Defn}
\begin{Remark} \label{rem:noquasireflections} If $(X,J,g,f)$ is an orbifold K\"ahler-Ricci shrinker with isolated singularities, then no isotropy group of $X$ contains quasi-reflections.
\end{Remark}

Suppose $(M^2,\widetilde{J},(\widetilde{g}_t)_{t\in [0,T)})$ is a K\"ahler-Ricci flow on a compact K\"ahler surface, which develops a singularity at time $T<\infty$. In the following, we recall compactness results proved in \cite[Theorem 2.37]{Bam3} and \cite[Theorem 2.8]{CHM}.

\begin{Theorem} \label{bamconvergence} Given any $\mu \in M_T$ and any sequence $t_i \nearrow T$, set $\widetilde{g}_{i,t}:=(T-t_i)^{-1}\widetilde{g}_{T+(T-t_i)t}$ and $\mu_t^i:=\mu_{T+(T-t_i)t}$ for $t\in [-(T-t_i)^{-1}T,0)$. After passing to a subsequence, there is a metric soliton $(\mathcal{X},(\nu_t)_{t\in (-\infty,0)})$ along with a correspondence $\mathfrak{C}$ such that we have the following $\mathbb{F}$-convergence within the correspondence on compact time intervals:
\begin{equation} \label{eq:Fconverge} (M,(\widetilde{g}_{i,t})_{t\in (-(T-t_i)^{-1}T,0)},(\mu_t^i)_{t\in (-(T-t_i)^{-1}T,0)})\xrightarrow[i\to \infty]{\mathbb{F},\mathfrak{C}} (\mathcal{X},(\nu_t)_{t\in (-\infty,0)}).\end{equation}
Moreover, the metric flow pair $(\mathcal{X},(\nu_{t})_{t\in (-\infty,0)})$ is modeled on an orbifold K\"ahler-Ricci shrinker with isolated singularities in the sense of \cite[Theorem 2.8]{CHM}.
\end{Theorem}

The following is a consequence of \cite[Lemma 2.8]{CHL} and \cite[Lemma 2.11]{CHM}.
\begin{Lemma} \label{lem:Hncenters}
\begin{enumerate}
    \item \label{lem:Hncenters1} There exists a unique map $\widehat{\pi}:M_T \to Y$ and some $C \in (0,\infty)$ such that for any $\mu \in M_T$, $t\in [0,T)$, and any $H_{4}$-center $(x,t)$ of $\mu$, 
\begin{equation*} d_{g_Y}(\widetilde{\pi}(x),\widehat{\pi}(\mu))\leq C\sqrt{T-t}.\end{equation*}
    \item \label{lem:Hncenters2} If $(X,g)$ is a tangent flow of $(M,(\widetilde{g}_t)_{t\in [0,T)})$ at $\mu \in M_T$, and $\psi_i:X\supseteq U_i \to M$ are open embeddings realizing the convergence, then for any compact set $L\subseteq X_{\operatorname{reg}}$ and any neighborhood $\mathcal{V}$ of $\widetilde{\pi}^{-1}(\widehat{\pi}(\mu))$, we have $\psi_i(L)\subseteq \mathcal{V}$ for $i=i(L,\mathcal{V}) \in \mathbb{N}$ sufficiently large.
\end{enumerate}
\end{Lemma}

Because $\liminf_{t \nearrow T} \operatorname{diam}_{\widetilde{g}_t}(M)>0$, we expect singularity models to be noncompact. The following will be sufficient for this purpose.

\begin{Lemma} \label{lem:noncompact} Suppose $Q_i \to \infty$, and either $\mu^i \in M_T$ or else $\mu^i=(\widetilde{\nu}_{x_i,t_i;t})_{t\in [0,t_i)}$ for some $(x_i,t_i)\in M\times (0,T)$ satisfying $t_i \nearrow T$. If $\mu^i \in M_T$, we write $t_i:=T$. Set
\begin{equation*}
    \widehat{g}_{i,t}:=Q_i\widetilde{g}_{t_i+Q_i^{-1}t}, \qquad \widehat{\mu}_t^i:= \mu_{t_i+Q_i^{-1}t}^i,
\end{equation*}
and suppose there is a metric soliton $(\mathcal{Z},(\nu_t)_{t\in (-\infty,0)})$, modeled on some orbifold K\"ahler-Ricci shrinker $(Z,g_Z)$, such that
\begin{equation*}
    (M,(\widehat{g}_{i,t})_{t\in [-Q_it_i,0)},(\widehat{\mu}_t^i)_{t\in [-Q_it_i,0)})\xrightarrow[i\to \infty]{\mathbb{F}} (\mathcal{Z},(\nu_t)_{t\in (-\infty,0)})
\end{equation*}
uniformly on compact time intervals. Then $Z$ is noncompact.
\end{Lemma}
\begin{proof} Suppose by way of contradiction that $Z$ is compact. It suffices to find $D<\infty$ and $x_i' \in M$ such that 
\begin{equation} \label{eq:diametercontradiction} M \times \{-2\}\subseteq P_{\widehat{g}_i}^{\ast}(x_i',-1;D,-2,0),\end{equation} 
since then \cite[Theorem 9.8]{Bam1} would imply
\begin{align*}
    \operatorname{vol}_{\widetilde{g}_{t_i-2Q_i^{-1}}}(M)&= Q_i^{-2}\operatorname{vol}_{\widehat{g}_{i,-2}}(M) \\&= Q_i^{-2}\operatorname{vol}_{\widehat{g}_{i,-2}}(P_{\widehat{g}_i}^{\ast}(x_i',-1;D,-2,0) \cap (M\times \{-2\})) \\
    & \leq C(T-(t_i-2Q_i^{-1}))^2,
\end{align*}
contradicting assumption \eqref{eq:volassumption}.

Fix $x_0 \in Z \cong \mathcal{Z}_{-1}$, and choose $D \in (1,\infty)$ such that $D \geq 100\sqrt{H_4}$ and $\mathcal{Z}_{[-3,-1]} \subseteq  P^{\ast}(x_0,-1,\frac{1}{10}D;-2,0)$. Choose $x_i \in M$ such that $(x_i,-1)\xrightarrow[i\to \infty]{\mathfrak{C}} (x_0,-1)$. Assume by way of contradiction that \eqref{eq:diametercontradiction} fails, so that (since $M$ is connected) there exist $y_i \in M$ such that 
\begin{equation} \label{eq:noncompactcontra} d_{W_1}^{\widehat{g}_{i,-3}}(\widehat{\nu}_{y_i,-2;-3}^i,\widehat{\nu}_{x_i,-1;-3}^i) =D.\end{equation} 
After passing to a subsequence, \cite[Lemma 4.2]{FH} gives a conjugate heat flow $(\mu_t)_{t\in (-\infty,-2)}$ such that 
\begin{equation*}
    (\widehat{\nu}_{y_i,-2;t}^i)_{t\in [-10,-2)} \xrightarrow[i\to \infty]{\mathfrak{C}} (\mu_t)_{t\in [-10,-2)},
\end{equation*}
where $\lim_{t \nearrow -2}\operatorname{Var}(\mu_t)=0$. We may therefore pass \eqref{eq:noncompactcontra} to the limit to get 
\begin{equation*} d_{W_1}^{\mathcal{Z}_{-3}}(\mu_{-3},\nu_{x_0,-1;-3})=D.
\end{equation*}
Letting $(z,-3)$ be an $H_4$-center of $\mu$, it then follows that
\begin{equation*}
    d_{W_1}^{\mathcal{Z}_{-3}}(\delta_z,\nu_{x_0,-1;-3})\geq D-\sqrt{H_4} \geq \frac{1}{2}D.
\end{equation*}
On the other hand, we have 
\begin{align*}
    d_{W_1}^{\mathcal{Z}_{-3}}(\delta_z,\nu_{x_0,-1;-3})<\frac{1}{10}D
\end{align*}
by assumption, yielding a contradiction. 
\end{proof}

\section{Regular fibers}

\label{section:regular}

In this section, we give a new proof that the tangent flow near a regular fiber is the cylinder $\mathbb{C}\times\mathbb{P}^1$. This was shown in \cite{XuZhang,JScollapsing} using different methods.

\begin{Proposition} \label{prop:tangentflowregularfiber} If $\mu \in M_T$ satisfies $\widehat{\pi}(\mu)\in U$, then the unique tangent flow of $(M,\widetilde{J},(\widetilde{g}_t)_{t\in [0,T)})$ at $\mu$ is the shrinking cylinder $\mathbb{C}\times \mathbb{P}^1$. 
\end{Proposition}
\begin{proof}
Suppose by way of contradiction that there exists a tangent flow $\mathcal{X}$ as in Theorem \ref{bamconvergence} at $\mu$ which is modeled on an orbifold K\"ahler-Ricci shrinker $(X,g)$ other than $\mathbb{C}\times \mathbb{P}^1$. By Lemma \ref{lem:noncompact}, $X$ is noncompact. If $X$ is the Gaussian shrinker on $\mathbb{C}^2$, then Perelman's pseudolocality gives $\lim_{t\nearrow T}\operatorname{vol}_{\widetilde{g}_t}(M)>0$, a contradiction. Thus $X$ is nontrivial. If $X$ has an orbifold singularity $x_0$, then there is a nontrivial subgroup $\Gamma \leq U(2)$ such that for all sufficiently small $r>0$, $\partial B_g(x_0,r)$ is diffeomorphic to $\mathbb{S}^3/\Gamma$, and such that if $H\subseteq T\partial B_g(x_0,r)\cong T(\mathbb{S}^3/\Gamma)$ is the image of the standard contact structure on $\mathbb{S}^3/\Gamma$, then $\omega$ restricts to a volume form on $H$ (see \cite[Proof of Proposition 5.4(i)]{CHM}).

By \cite[Theorem 9.31]{Bam2}, there exist open embeddings $\psi_i:X_{\operatorname{reg}}\supseteq U_i \to M$ realizing the convergence
\begin{equation*}
    (T-t_i)^{-1}\psi_i^{\ast}\widetilde{g}_{t_i} \to g, \qquad \psi_i^{\ast}\widetilde{J}\to J
\end{equation*}
in $C_{\operatorname{loc}}^{\infty}(X_{\operatorname{reg}})$. It follows that $\widetilde{\omega}_{t_i}|_{(\psi_i)_{\ast}H}$ is a volume form on the sub-bundle $(\psi_i)_{\ast}H$ of $T\Sigma_{i,r}$, where $\Sigma_{i,r}:=\psi_i(\partial B_g(x_0,r))$. Because $\widehat{\pi}(\mu)\in U$, \cite[Proposition 1-4-5]{matsuki} gives a neighborhood $\mathcal{V}$ of $\widetilde{\pi}^{-1}(\widehat{\pi}(\mu))$ biholomorphic to $\mathbb{D}\times \mathbb{P}^1$ such that $\Sigma_{i,r} \subseteq \mathcal{V}$ for all $i=i(r)\in \mathbb{N}$ sufficiently large. Because $\mathcal{V}$ is simply connected and has one end, it follows that $\mathcal{V}\setminus \Sigma_{i,r}$ has exactly two connected components, one of which is bounded. Call the bounded component $\mathcal{B}_{i,r}$. Because the intersection form of $H_2(\mathcal{V},\mathbb{Z})$ is zero, it follows that $\mathcal{B}_{i,r}$ is a minimal symplectic filling of the standard contact structure on $\mathbb{S}^3/\Gamma$. Because $\mathcal{V}$ is diffeomorphic to an open subset of $\mathbb{S}^4$, it follows from \cite{crisp} that $\Gamma$ is the Quaternion group $Q_8$. In particular, $\mathbb{C}^2/\Gamma$ is a simple singularity, so by \cite{OO}, $\mathcal{B}_{i,r}$ contains a 2-sphere of self-intersection $-2$, yielding a contradiction. 

We conclude that $X$ is smooth, so by \cite{CCD,BCCD,LiWang}, $X$ is biholomorphic to the BCCD shrinker or the FIK shrinker. In either case, $X$ has a $(-1)$-curve, which by the argument of \cite{CCD} implies that $\mathcal{V}$ contains a $(-1)$-curve as well, again yielding a contradiction. 
\end{proof}

\section{Singular fibers}
\label{section:singular}

Suppose $(X,J,g)$ is a tangent flow of $(M,\widetilde{J},(\widetilde{g}_t)_{t \in [0,T)})$ based at some $\mu \in M_T$ with $\widehat{\pi}(\mu)\in Y\setminus U$. Set $y_0:= \widehat{\pi}(\mu)$. Letting $t_i$ be as in \eqref{eq:Fconverge}, we fix a precompact open exhaustion $(U_i)$ of $X_{\operatorname{reg}}$ and open embeddings $\psi_i : U_i \to M$ as in \cite[Theorem 9.31]{Bam2}, so that $\widetilde{g}_{i,t}:=(T-t_i)^{-1}\widetilde{g}_{T+(T-t_i)t}$ satisfies
\begin{equation} \label{eq:smoothconvergence}
\psi_i^{\ast}\widetilde{g}_{i,t} \to g, \qquad \psi_i^{\ast}\widetilde{J} \to J \qquad \text{in }   C_{\operatorname{loc}}^{\infty}(X_{\operatorname{reg}}\times (-\infty,0))
\end{equation}
as $i\to \infty$. 

\subsection{Topology near a singular fiber}
\label{subsection:topology}

In this subsection, we show that a neighborhood of any singular fiber of $\widetilde{\pi}$ is biholomorphic to an open subset of the blowup $M_{\operatorname{BCCD}}$ of $\mathbb{C}\times \mathbb{P}^1$ at a point.

\begin{Lemma} \label{lem:singularfibertopology} There is a neighborhood $\mathcal{V}$ of $\widetilde{\pi}^{-1}(y_0)$ and a biholomorphism $\mathcal{V} \to \mathcal{V}_{\operatorname{BCCD}}$, where $\mathcal{V}_{\operatorname{BCCD}}$ is the preimage of the unit disk under the map $M_{\operatorname{BCCD}}\to \mathbb{C}$. 
\end{Lemma}
\begin{proof}
Let 
\begin{equation*} \widetilde{\pi}^{-1}(y_0)=F=\sum_{i=1}^m a_i E_i \end{equation*}
be the decomposition into prime divisors $E_i$, where $a_i$ are positive integers. Because $-K_M$ is $\widetilde{\pi}$-ample, 
\begin{equation*}
    K_M \cdot E_i <0, \qquad F\cdot E_i=0 \qquad \text{ for } 1\leq i\leq m,
\end{equation*}
so we can compute
\begin{equation} \label{eq:basicintersection}
    0=F\cdot E_i =\sum_{j=1}^m a_jE_j \cdot E_i \geq a_i E_i \cdot E_i.
\end{equation}
Letting $F'$ be a regular fiber of $\widetilde{\pi}$, the adjunction formula gives 
\begin{equation} \label{eq:adjunctionconsequence}
   -2= F'\cdot K_M = \sum_{i=1}^m a_i E_i \cdot K_M
\end{equation}
since $F,F'$ are numerically equivalent. Because $a_i E_i \cdot K_M \leq -1$ for each $i\in \{1,...,m\}$, it follows that $m\leq 2$. If $m=1$, then $E_1 \cdot E_1=0$ and adjunction imply
\begin{equation*}
    2p_a(E_1)-2=E_1 \cdot K_M,
\end{equation*}
where the arithmetic genus $p_a(E_1)$ is a nonnegative integer. It follows that $p_a(E_1)=0$ and $E_1\cdot K_M=-2$, so that \eqref{eq:adjunctionconsequence} implies $a_1=1$. By \cite[Proposition 1-4-5]{matsuki}, $y_0$ is a regular value of $\widetilde{\pi}$, contradicting $y_0 \in Y\setminus U$. Thus $m=2$, and \eqref{eq:adjunctionconsequence} yields $a_1=a_2=1$ and $E_1 \cdot K_M = E_2 \cdot K_M=-1$. By the adjunction formula,
\begin{equation*}
    2p_a(E_i)-2=-1+E_i \cdot E_i,
\end{equation*}
hence $p_a(E_i)=0$ and $E_i \cdot E_i=-1$. In other words, both $E_i$ are $(-1)$-curves. Finally, \eqref{eq:basicintersection} gives $E_1 \cdot E_2=1$.

Choose a holomorphic disk $\mathbb{D}_{y_0} \subseteq Y$ around $y_0$ with $U\cap \mathbb{D}_{y_0}=\mathbb{D}_{y_0}\setminus \{y_0\}$, and set $\mathcal{V}:=\widetilde{\pi}^{-1}(\mathbb{D}_{y_0})$. Let $\rho:\mathcal{V}\to \mathcal{W}$ be the blowdown along $E_1$, so that $\rho(E_2) \cong \mathbb{P}^1$ has self-intersection $0$, and $\widetilde{\pi}$ factors as $\widetilde{\pi}=\widehat{\pi}\circ \rho$ for some holomorphic map $\widehat{\pi}:\mathcal{W}\to \mathbb{D}_{y_0}$. By \cite[Proposition 1-4-5]{matsuki}, after possibly shrinking $\mathbb{D}_{y_0}$, $\widehat{\pi}|_{\widehat{\pi}^{-1}(\mathbb{D}_{y_0})}$ is biholomorphic to the projection map $\mathbb{D}_{y_0}\times \mathbb{P}^1 \to \mathbb{D}_{y_0}$, hence the claim follows. 
\end{proof}

\subsection{Fano fibration structure}
\label{subsection:fanofibration}

Because $X$ is non-compact, it is not a Fano variety in the usual sense. However, it is known \cite{HZ} that $X$ admits a fibration over an affine variety, whose fibers are Fano. In our situation, a much more elementary proof of this fact is possible, which we now give.

\begin{Proposition} $X$ admits the structure of a polarized Fano fibration $\pi:X\to \mathcal{C}$ for some normal affine algebraic variety $\mathcal{C}$.
\end{Proposition}
\begin{proof}
    By \cite{Bam3}, we know that $(X,J,g)$ is a smooth orbifold metric, possibly with infinitely many orbifold points escaping to infinity. Let $\mathbb{T}$ be the real torus generated by $J\nabla f$ in the isometry group of $X$. Generalizing the results in \cite[Theorem~4.1]{HNP} to the orbifold setting, we conclude that the moment map 
\[
\mu : X \to \mathrm{Lie}(\mathbb{T})^*
\]
is proper and open onto its image, and its image is a closed convex set. Then as in \cite[Lemma 3.3]{SunZhang}, we can perturb the vector field $J\nabla f$ to get a holomorphic Killing vector field generating an $\mathbb{S}^1$-action, which moreover admits a proper and bounded below Hamiltonian potential $u \in C^{\infty}(X)$ normalized so that
\begin{equation} \label{eq:laplacian}
    \Delta_g u - \mathcal{L}_{\frac{1}{2}\nabla^g f}u=-u.
\end{equation}

Because $\mu$ is open onto its image, the same is true for $u$. In particular, $u$ admits no local maxima, and every local minimum is necessarily a global minimum. Therefore, outside of a compact set, all critical points of $u$ must be isolated. 

Suppose by way of contradiction that there are infinitely many isolated critical points $x_j$, so that
\begin{equation}
    z_j:=u(x_j)\rightarrow \infty \text{ as }j\rightarrow \infty.
\end{equation}
At a critical point $x_j$, \eqref{eq:laplacian} implies the relation
\begin{equation}
    p_j - q_j = -z_j,
\end{equation}
where $p_j$ and $q_j$ denote the weights of the induced $S^1$-action on the tangent space $T^{1,0}_{x_j}X$. If $x_j$ is an orbifold point, this relation is understood in a local uniformizing chart. Because $J\nabla u$ generates an $\mathbb{S}^1$-action by biholomorphisms, near each critical point $x_j$ of $u$, we can find local holomorphic coordinates $(w_1, w_2)$ centered at $x_j$ such that the $S^1$-action is given by 
\begin{equation}
    e^{\ii \theta}(w_1,w_2)=(e^{\ii p_j\theta}w_1,e^{-\ii q_j\theta}w_2)
\end{equation}
and the Hamiltonian potential $u$ has the form 
\begin{equation}
    u=z_j+p_j|w_1|^2-q_j|w_2|^2+O(|w|^3).
\end{equation}
Because the $\mathbb{S}^1$-action is effective, we know $\gcd(p_j,q_j)=1$. Again, if $x_j$ is an orbifold point, the above relation holds on a local uniformizing chart. Let $\Gamma_j$ denote the local uniformizing group near the orbifold point $x_j$. Since $\Gamma_j$ commutes with the $S^1$-action and $\gcd(p_j,q_j)=1$, it follows that $\Gamma_j$ is generated by $\mathrm{diag}(\zeta_1,\zeta_2)$, where $\zeta_1$ and $\zeta_2$ are primitive $|\Gamma_j|$-th roots of unity. 

In the following, let $\Gamma_p$ denote the stabilizer of a point $p\in X$ under the $S^1$-action.
Then by considering points near $x_j$ with local coordinates $(\epsilon,0)$ and $(0,\epsilon)$, we can find points $x_j',y_j'$ in a neighborhood of $x_j$ such that 
\begin{equation} \label{eq:Fanofibration1}
    u(x_j')<z_j<u(y_j')
\end{equation}
and for $k_j=|\Gamma_j|\geq 1$,
\begin{equation} \label{eq:Fanofibration2}
    |\Gamma_{x_j'}|\geq |\Gamma_{y_j'}|+k_j|z_j|.
\end{equation}
By considering the broken backwards gradient flows of $u$ starting at $x_j$, and using \eqref{eq:Fanofibration1},\eqref{eq:Fanofibration2} along with the fact that gradient flows of $u$ preserve the $\mathbb{S}^1$-isotropy groups, we obtain a sequence $\widetilde{x}_j$ in a fixed compact set such that
\begin{equation}
    |\Gamma_{\tilde x_j}|\rightarrow \infty,
\end{equation}
yielding a contradiction. Thus the set of critical points of $u$ is compact. Because any orbifold point of $X$ is a critical point of $u$, this means $X$ has at most finitely many orbifold points. 

From here, one may construct a polarized Fano fibration $\pi:X\to \mathcal{C}$ exactly as in \cite[pp.~7--8]{SunZhang}, where $\mathcal{C}$ is a normal affine variety.
\end{proof}

To identify $\mathcal{C}$, we first require a more detailed description of the asymptotic geometry of $X$, which is provided by the following.

\begin{Proposition} \label{prop:curvaturenottozero}
\begin{enumerate}[(i)]
    \item \label{prop:curvaturenottozero0} $X$ has bounded curvature.

    \item \label{prop:curvaturenottozero1} There is a sequence $x_i \in X$ diverging to infinity such that $\liminf_{i \to \infty} |\operatorname{Rm}_g|_g(x_i)>0$.

    \item \label{prop:curvaturenottozero2} For any compact subset $A \subseteq X$, there is a compact holomorphic curve $D \subseteq  X\setminus A$ of trivial self-intersection which is isomorphic to $\mathbb{P}^1$. 
\end{enumerate}

\end{Proposition}

\begin{proof}
\eqref{prop:curvaturenottozero0} This follows from \cite[Theorem 4.2]{CHM}.

\eqref{prop:curvaturenottozero1}
Suppose by way of contradiction that such a sequence does not exist. In this case, \cite[Proof of Proposition 5.4(ii)]{CHM} gives $C<\infty$ such that 
\[
|\operatorname{Rm}_g|_{g}(x)\leq\frac{C}{d_g^{2}(x_{\infty},x)+1}
\]
for all $x\in X$, where $x_{\infty}\in X$ is fixed. In particular, for $\Lambda>0$ large and to be
determined, there exists $D=D(\Lambda)>0$ such that $X\setminus X_{\operatorname{reg}} \subseteq B_g(x_{\infty},\frac{D}{2})$ and
\[
\sup_{B_{g}(x,2\Lambda)}|\operatorname{Rm}_g|_{g}\leq\frac{1}{(2\Lambda)^{2}}
\]
for all $x\in X\setminus B_{g}(x_{\infty},D)$. Because $X$ is $\kappa$-noncollapsed
at all scales for some $\kappa>0$, this implies 
\[
\frac{\text{vol}_{g}(B_{g}(x,\frac{1}{2}\Lambda))}{\Lambda^{4}}\geq\kappa
\]
for all $x\in X\setminus B_{g}(x_{\infty},D)$. Thus, for any $x\in\mathcal{A}:=B_{g}(x_{\infty},3D)\setminus\overline{B}_{g}(x_{\infty},2D)$, it follows that $\widetilde{g}_i:=(T-t_i)^{-1}\widetilde{g}_{t_i}$ satisfies 
\[
\sup_{B_{\widetilde{g}_{i}}(\psi_{i}(x),\Lambda)}|\operatorname{Rm}_{\widetilde{g}_{i}}|_{\widetilde{g}_{i}}\leq\frac{1}{\Lambda^{2}},\,\,\,\,\,\,\,\,\,\,\,\,\,\,\frac{\text{vol}_{\widetilde{g}_{i}}(B_{\widetilde{g}_{i}}(\psi_{i}(x),\Lambda))}{\Lambda^{4}}\geq\kappa
\]
for $i=i(\Lambda)\in \mathbb{N}$ sufficiently large. Equivalently,
\[
\sup_{B_{\widetilde{g}_{t_{i}}}(\psi_{i}(x),\Lambda\sqrt{T-t_{i}})}|\operatorname{Rm}_{\widetilde{g}_{t_{i}}}|_{\widetilde{g}_{t_{i}}}\leq\frac{1}{\Lambda^{2}(T-t_i)},\,\,\,\,\,\,\,\,\,\,\,\,\,\,\frac{\text{vol}_{\widetilde{g}_{t_{i}}}(B_{\widetilde{g}_{t_{i}}}(\psi_{i}(x),\Lambda\sqrt{T-t_{i}}))}{(\Lambda\sqrt{T-t_{i}})^{4}}\geq\kappa
\]
whenever $i=i(\Lambda)\in \mathbb{N}$ is large. We can therefore apply \cite[Theorem 1.2]{LuPseudo}
at scale $r_{0}=\Lambda\sqrt{T-t_{i}}$ to obtain $\epsilon_{0}=\epsilon_{0}(\kappa)>0$
such that 
\[
|\operatorname{Rm}_{\widetilde{g}_{t}}|_{\widetilde{g}_{t}}(\psi_{i}(x),t)\leq\frac{1}{\epsilon_{0}^{2}\Lambda^{2}(T-t_{i})}
\]
for all $x\in\mathcal{A}$ and $t\in[t_{i},\min\{t_{i}+\epsilon_{0}^{2}\Lambda^{2}(T-t_{i}),T\})$.
If we choose $\Lambda\geq2\epsilon_{0}^{-1}$, then $t_{i}+\epsilon_{0}^{2}\Lambda^{2}(T-t_{i})>T$,
hence
\[
\sup_{t\in[t_{i},T)}\sup_{y\in\psi_{i}(\mathcal{A})}|\operatorname{Rm}_{\widetilde{g}_{t}}|_{\widetilde{g}_{t}}(y)\leq\frac{1}{4(T-t_{i})}
\]
whenever $i=i(\Lambda)\in \mathbb{N}$ is sufficiently large. Fixing such $i$, we set $\widetilde{\mathcal{A}}:=\psi_{i}(\mathcal{A})$ and integrate
\[
\frac{d}{dt}\log\text{vol}_{\widetilde{g}_{t}}(\widetilde{\mathcal{A}})=-\frac{1}{\text{vol}_{\widetilde{g}_{t}}(\widetilde{\mathcal{A}})}\int_{\widetilde{\mathcal{A}}}R_{\widetilde{g}_{t}}\frac{1}{2}\widetilde{\omega}_{t}^{2}\geq-\frac{C}{T-t_{i}}
\]
from time $t_{i}$ to time $t$ to obtain
\[
\inf_{t\in[t_{i},T)}\text{vol}_{\widetilde{g}_{t}}(M)\geq\inf_{t\in[t_{i},T)}\text{vol}_{\widetilde{g}_{t}}(\widetilde{\mathcal{A}})\geq e^{-C}\text{vol}_{\widetilde{g}_{t_{i}}}(\widetilde{\mathcal{A}})>0,
\]
a contradiction.

\eqref{prop:curvaturenottozero2} By \eqref{prop:curvaturenottozero0}, \eqref{prop:curvaturenottozero1}, and \cite[Proof of Proposition 3.1]{CCD}, there is a sequence $x_i \to \infty$ such that the sequence $(X,J,g,x_i)$ converges in the smooth pointed Cheeger-Gromov sense to the cylindrical K\"ahler metric on $\mathbb{C}\times \mathbb{P}^1$, so the claim follows from \cite[Corollary 2.3]{CCD}.
\end{proof}

\begin{Proposition} $\mathcal{C}$ is equivariantly isomorphic to $\mathbb{C}$ equipped with a linear torus action.
\end{Proposition}
\begin{proof}
    By Lemma \ref{lem:noncompact}, $\mathcal{C}$ is not a point. If $\dim_{\mathbb{C}}(\mathcal{C})=2$, then $\pi$ restricts to a biholomorphism $X\setminus \pi^{-1}(A)\cong \mathcal{C} \setminus A$ for some compact subset $A \subseteq \mathcal{C}$. It follows that $X$ is 1-convex, contradicting Proposition \eqref{prop:curvaturenottozero}\eqref{prop:curvaturenottozero2}. Thus $\dim_{\mathbb{C}}(\mathcal{C})=1$. Because $\mathcal{C}$ is a normal complex curve, it must be smooth. By \cite[paragraph before Lemma 2.1]{CarrKutt}, any smooth affine curve which admits an attractive torus action must be $\mathbb{C}^{\ast}$-equivariantly isomorphic to $\mathbb{C}$ equipped with a linear torus action. 
\end{proof}

\subsection{Properties of the central fiber}
\label{subsection:centralfiber}

Let $\mathbb{T}^{\mathbb{C}}$ be the complexification of the torus generated by $J\nabla f$. Since the generic fiber of $\pi:X\to \mathbb{C}$ is reduced and isomorphic to $\mathbb{P}^1$, and because $\pi$ is $\mathbb{T}^{\mathbb{C}}$-equivariant for some attractive torus action on $\mathbb{C}$, it follows that $X\setminus \pi^{-1}(0)$ is biholomorphic to $\mathbb{P}^1\times \mathbb{C}^{\ast}$, where $\mathbb{C}^{\ast}$ is equipped with the $\mathbb{T}^{\mathbb{C}}$-action induced by the action on the base $\mathbb{C}$. 

In this subsection, we describe the scheme-theoretic central fiber 
\begin{equation*}
    \pi^{-1}(0)=E=\sum_{j=1}^Q a_j E_j,
\end{equation*}
where the $E_j$ are the irreducible components of $\pi^{-1}(0)$, and $a_j$ are positive integers. 
We let $\rho: \widetilde{X} \to X$ denote the minimal resolution of $X$, so that $\rho$ restricts to an isomorphism over $X\setminus E$. 

\begin{Lemma} \label{lem:curvenumbers}
\begin{enumerate}[(i)]

    \item \label{lem:curvenumbers1} $E_i\cdot E=0$, $E_i \cdot E_i\leq 0$, and $E_i \cdot E_j \geq 0$ for $i,j=1,...,Q$ with $i\neq j$. Moreover, $E_i \cdot E_i=0$ if and only if $Q=1$. 
    
    \item \label{lem:curvenumbers2} If $Q\geq 2$, then the proper transform $\widetilde{E}_i$ of each $E_i$ is a $(-1)$-curve. If $Q=1$, then the proper transform of $E_1$ is a $\mathbb{P}^1$ with self-intersection $0$ or $-1$.
\end{enumerate}
\end{Lemma}
\begin{proof} 
\eqref{lem:curvenumbers1} For $1\leq i \neq j\leq Q$, $E_i,E_j$ do not contain a common irreducible component, so that $E_i \cdot E_j \geq 0$. Because $E_i$ is disjoint from any regular fiber of $\pi$, it follows that $E_i \cdot E=0$, hence
\begin{align*}
    0= E_i \cdot \sum_{j=1}^Q a_j E_j \geq a_i E_i \cdot E_i,
\end{align*}
with equality if and only if $E_i \cdot E_j =0$ for every $j\in \{1,...,Q\} \setminus \{i\}$. That is, if equality holds, $E_i$ is a connected component of $E$; however, the fibers of $\pi$ are connected, so in this case, $Q=1$. 

\eqref{lem:curvenumbers2} Arguing as in \cite[Proposition 5.7]{CHM}, we have $K_{\widetilde{X}}\cdot \widetilde{E}_i <0$ for $1\leq i\leq Q$. Because $\widetilde{X}\to X\to \mathbb{C}$ still has connected fibers, we can argue as in \eqref{lem:curvenumbers1} (with $X$ replaced by $\widetilde{X}$) to conclude $\widetilde{E}_i \cdot \widetilde{E}_i \leq 0$, with strict inequality unless $Q=1$. The claim then follows by the adjunction formula.
\end{proof}

We now show that in fact there are at most $2$ irreducible components of $E$.

\begin{Proposition} \label{prop:components}
   $Q\leq 2$.
\end{Proposition}
\begin{proof}
 Let $X_{\sing}=\{p_i\}_{i=1}^P$ denote the set of orbifold points of $X$. Let $V$ be a small neighborhood of  $X_{\sing}$ and $U$ an open subset of $X$, compactly contained in $X\setminus X_{\sing}$ such that $U\cap V$ is diffeomorphic to $\bigsqcup_{i=1}^P (S^3/\Gamma_i \times (-1,1))$, where $\Gamma_i$ is the isotropy group of $X$ at $p_i$. We may also assume $\cup_{i=1}^Q E_i \subseteq U \cup V$. By the Mayer-Vietoris sequence for $U$ and $V$, the following is exact:
    \begin{equation}
        H_2(U\cap V,\mathbb{R})\rightarrow H_2(U,\mathbb{R})\oplus H_2(V,\mathbb{R})\rightarrow H_2(U\cup V,\mathbb{R})\rightarrow H_1(U\cap V,\mathbb{R}).
    \end{equation}
    Because $U\cap V$ is a finite disjoint union of rational homology spheres and $V$ is a finite disjoint union of contractible sets, it follows that the map $H_2(U,\mathbb{R})\to H_2(U\cup V,\mathbb{R})$ is an isomorphism. We may therefore choose 2-cycles $\sigma_1,...,\sigma_Q$ in $U$ which are homologous in $U\cup V$ to $E_1,...,E_Q$, respectively. By \cite[proof of Proposition 5.10(i)]{CHM}, the topological intersection number $\sigma_i \cdot \sigma_j$ in $U$ is moreover equal to the intersection number $E_i \cdot E_j \in \mathbb{Q}$ in the sense of \cite{mumford}. Let $W \subseteq H_2(U,\mathbb{R})$ be a maximal negative-definite subspace of $\operatorname{span}\{\sigma_1,...,\sigma_Q\}$.     
    By \cite[Lemma I.2.10]{BKV} applied to the matrix $(-\sigma_i \cdot \sigma_j)_{i,j=1}^Q$, it follows that $\dim(W)\geq Q-1$.

    By Lemma \ref{lem:singularfibertopology}, there is a neighborhood $\mathcal{V}$ of $\widetilde{\pi}^{-1}(\widehat{\pi}(\mu))$ in $M$ and an open holomorphic embedding $\iota:\mathcal{V} \to M_{\operatorname{BCCD}}$. By Lemma \ref{lem:Hncenters}\eqref{lem:Hncenters2}, we have $\psi_i(\overline{U})\subseteq \mathcal{V}$ for sufficiently large $i\in \mathbb{N}$.  It follows that $\iota \circ \psi_i:U \to M_{\operatorname{BCCD}}$ is an orientation-preserving open embedding for sufficiently large $i\in \mathbb{N}$. Any nonzero $\sigma \in W$ satisfies
    \begin{equation*}
        (\iota \circ \psi_i)_{\ast}\sigma \cdot (\iota \circ \psi_i)_{\ast}\sigma = \sigma \cdot \sigma<0,
    \end{equation*}
    so that in particular $(\iota \circ \psi_i)_{\ast}\sigma \neq 0$. This implies that $(\iota \circ \psi_i)_{\ast}W$ is a subspace of $H_2(M_{\operatorname{BCCD}},\mathbb{R})$ of dimension at least $Q-1$ on which the intersection pairing is negative definite. However, this subspace of $H_2(M_{\operatorname{BCCD}},\mathbb{R})$ is at most 1-dimensional, hence $Q-1\leq 1$. 
\end{proof}

\subsection{Identification of the relative minimal model}

\label{subsection:relativeMM}

Recall that $\rho:\widetilde{X}\to X$ is the minimal resolution of $X$. In this subsection, we show that $\widetilde{X}\to \mathbb{C}$ is obtained by iteratively blowing up $\mathbb{C}\times \mathbb{P}^1$ at points in the central fiber.

\begin{Lemma} \label{lem:relativeMM}
    There is a birational morphism $\rho_{\min}:\widetilde{X}\to \mathbb{C}\times \mathbb{P}^1$ making the following diagram commute:
    \[
\begin{tikzcd}[row sep=large, column sep=large]
& \widetilde{X} \arrow[dl, "\rho_{\min}"'] \arrow[dr, "\rho"] & \\
\mathbb{C}\times \mathbb{P}^1 \arrow[dr] 
&& X \arrow[dl, "\pi"] \\
& \mathbb{C} &
\end{tikzcd}
.\]
\end{Lemma}
\begin{proof}
  Let $\rho_{\min}:\widetilde{X}\to \widetilde{X}_{\rm min}$ be obtained from $\widetilde{X}$ by successively contracting vertical $(-1)$-curves until there are no $(-1)$-curves on $\widetilde{X}_{\rm min}$. Note that
$(-1)$-curves are always over the origin, so the process terminates and there is a holomorphic fibration $\pi_{\min}:\widetilde{X}_{\rm min}\to \mathbb{C}$ satisfying $\pi_{\min}\circ \rho_{\min}= \pi \circ \rho$, whose general fiber is $\IP^1$. Let $F=\sum_i m_i C_i$ denote the scheme-theoretic fiber $\pi_{\min}^{-1}(0)$, where $m_i$ are positive integers. Then we have 
  \(C_i\cdot F=0.\)
 Flatness of $\pi_{\min}$ gives $K\cdot F=-2$, where
$K:=K_{\widetilde{X}_{\rm min}}$. Suppose $F$ had more than one component. Since $F$ is
connected, each $C_i$ meets some $C_j$ with $j\neq i$, hence
 $C_i^2<0$ for all $i$. On the other hand $\sum_i m_i(K\cdot C_i)=-2<0$ forces
$K\cdot C_{i_0}\le-1$ for some $i_0$. Letting $p_a(C_{i_0})$ denote the arithmetic genus of $C_{i_0}$,  adjunction then gives
\[
  2p_a(C_{i_0})-2=C_{i_0}^2+K\cdot C_{i_0}\le-2 ,
\]
so that $p_a(C_{i_0})=0$, hence $C_{i_0}\cong \mathbb{P}^1$ and $C_{i_0}^2=K\cdot C_{i_0}=-1$. Then $C_{i_0}$ is a $(-1)$-curve in a fiber,
contradicting minimality. So $F=mC$ with $C^2=0$, $K\cdot C=-2/m$. Adjunction then gives $p_a(C)=1-\frac{1}{m}\in \mathbb Z_{\geq 0}$. Therefore $m=1$ and $C\cong \mathbb{P}^1$, so $\pi_{\min}$ has no singular fibers. Because any such fibration over $\mathbb{C}$ is trivial, it follows that $\widetilde{X}_{\rm min}$ is biholomorphic to $\mathbb P^1\times \mathbb C$.
\end{proof}

\subsection{All singularities are Wahl}
\label{subsection:Wahl}

In this subsection, we show that if $X$ has any orbifold singularities, they must belong to a specific class of cyclic quotient singularities called Wahl singularities. 

Suppose $x\in X$ is an orbifold point, and let $\Sigma_r :=\partial B_g(x,r)$ be a small geodesic sphere around $x$, which is diffeomorphic to $\mathbb{S}^3/\Gamma$ for some finite subgroup $\Gamma \leq U(2)$ acting freely on $\mathbb{S}^3$. Define 
\begin{equation*} A'(r):=\overline{B}_g(x,2r)\setminus B_g(x,\frac{r}{4}), \qquad A(r):=\overline{B}_g(x,r)\setminus B_g(x,\frac{r}{2}).
\end{equation*}

\begin{Lemma} \label{lem:primitiveonlimit} For $r>0$ sufficiently small, there exists a 1-form $\alpha \in \mathcal{A}^1(A'(r))$ such that $d\alpha = \omega|_{A'(r)}$ and $\sup_{A'(r)}|\alpha|_g \leq C(\Gamma)r$.
\end{Lemma}
\begin{proof} Take $\alpha =d^c \varphi$ for a K\"ahler potential $\varphi$ of $\omega$ on $B_g(x,2r)$, constructed as in \cite[Section 7.2]{Gaborbook} after pulling back to a local uniformizing chart and averaging over the action of $\Gamma$.
\end{proof}

Let $\psi_i$ be as in \eqref{eq:smoothconvergence}.
For $i=i(r)\in \mathbb{N}$ sufficiently large, we may therefore define 
\begin{equation*} \Sigma_{i,r}:=\psi_i(\partial B_g(x,r)), \quad A_{i,r}=\psi_i(A(r)),\quad A_{i,r}'=\psi_i(A'(r)).\end{equation*}
By Lemma \ref{lem:Hncenters}, we have $\Sigma_{i,r} \subseteq \mathcal{V}$ for sufficiently large $i=i(r)\in \mathbb{N}$. Because $\mathcal{V}$ is simply connected, it follows that $\mathcal{V}\setminus \Sigma_{i,r}$ has exactly two connected components. We let $\mathcal{B}_{i,r}$ denote the closure of the component containing $\psi_i(\partial B_g(x,\frac{r}{2}))$.

\begin{Lemma} \label{lem:intersectionnegative} The intersection form on $H_2(\mathcal{B}_{i,r},\mathbb{R})$ is negative definite, and $b_2(\mathcal{B}_{i,r})\leq 1$. 
\end{Lemma}
\begin{proof} The intersection form is nondegenerate by the long exact sequence of the pair $(\mathcal{B}_{i,r},\Sigma_{i,r})$, since $\partial \mathcal{B}_{i,r}=\Sigma_{i,r}$ is a rational homology sphere and hence $H_2(\Sigma_{i,r}, \mathbb{R}) = 0$. Because the intersection number is preserved by the map on homology $\iota_{\ast}:H_2(\mathcal{B}_{i,r},\mathbb{R})\to H_2(\mathcal{V},\mathbb{R})$ induced by the inclusion $\iota:\mathcal{B}_{i,r}\hookrightarrow \mathcal{V}$, if $\sigma \in \ker(\iota_{\ast})$, then $\sigma\cdot \sigma'=0$ for all $\sigma'\in H_2(\mathcal{B}_{i,r},\mathbb{R})$. By the non-degeneracy of the intersection form on $H_2(\mathcal{B}_{i,r},\mathbb{R})$, it follows that $\sigma=0$. Thus $\iota_{\ast}$ is injective. Because the intersection form of $H_2(\mathcal{V},\mathbb{R})$ is non-positive and has rank 1, the claim follows.
\end{proof}

\begin{Lemma} \label{lem:primitivealongsequence} For $i=i(r)\in \mathbb{N}$ sufficiently large, there exists $\alpha_{i,r}\in \mathcal{A}^1(A_{i,r})$ satisfying $d\alpha_{i,r}=\widetilde{\omega}_{i,-1}|_{A_{i,r}}$ and $|\alpha_{i,r}|_{\widetilde{g}_{i,-1}}\leq Cr$. 
\end{Lemma}
\begin{proof}
Because $A'(r)$ deformation retracts to the rational homology sphere $\mathbb{S}^3/\Gamma$, it follows that $H^2(A'(r),\mathbb{R})=0$.
Thus $\psi_i^{\ast}(\widetilde{\omega}_{i,-1}|_{A_{i,r}'})$ is a sequence of $d$-exact forms on $A'(r)$ converging smoothly to $\omega|_{A'(r)}$, so the claim follows from Lemma \ref{lem:primitiveonlimit} and \cite[Lemma 3.2.1]{SchwarzBook} applied to $\psi_i^{\ast}(\widetilde{\omega}_{i,-1}|_{A_{i,r}'})-\omega|_{A'(r)}$. 
\end{proof}

Note that $[\widetilde{\omega}_{i,t}|_{\mathcal{V}}] = (T-t_i)^{-1}[\widetilde{\omega}_{T}|_{\mathcal{V}}] - tc_1(\mathcal{V})$, as inclusion is a linear map on cohomology. Since $[\widetilde{\omega}_T] =\widetilde{\pi}^{\ast}[\omega_Y]$ for some K\"ahler metric $\omega_Y$ on $Y$, and any such metric satisfies $[\omega_Y|_{\widetilde{\pi}(\mathcal{V}) }] = 0$ since $\widetilde{\pi}(\mathcal{V})\cong \mathbb{D}$, we obtain $[\widetilde{\omega}_{i,-1}|_{\mathcal{V}}]=c_1(\mathcal{V})$. 

Let $j:H^2(\mathcal{B}_{i,r},\partial \mathcal{B}_{i,r},\mathbb{R})\to H^2(\mathcal{B}_{i,r},\mathbb{R})$ be the natural map, and let $$[\mathcal{B}_{i,r},\partial \mathcal{B}_{i,r}] \in H_4(\mathcal{B}_{i,r},\partial \mathcal{B}_{i,r},\mathbb{R})$$ be the fundamental class. For $\alpha,\beta \in H^2(\mathcal{B}_{i,r},\mathbb{R})$ define
\begin{equation*}
    q(\alpha,\beta):= \langle j^{-1}(\alpha)\cup \beta, [\mathcal{B}_{i,r},\partial \mathcal{B}_{i,r}]\rangle.
\end{equation*}

\begin{Lemma} \label{lem:boundaryintegral} With respect to the de Rham isomorphism, 
\begin{equation*}
q([\eta],[\eta]):=\int_{\mathcal{B}_{i,r}}\eta \wedge \eta - \int_{\Sigma_{i,r}} \theta \wedge \eta, 
\end{equation*}
for any closed 2-form $\eta \in \mathcal{A}^2(\mathcal{B}_{i,r})$ and $\theta \in \mathcal{A}^1(A_{i,r})$ such that $\eta|_{A_{i,r}}=d\theta$. Moreover, $q([\eta],[\eta])\leq 0$, with equality if and only if $\eta$ is $d$-exact.
\end{Lemma}
\begin{proof} 

With respect to the de Rham isomorphisms, the inverse of the natural map $j:H^2(\mathcal{B}_{i,r},\partial \mathcal{B}_{i,r},\mathbb{R})\to H^2(\mathcal{B}_{i,r},\mathbb{R})$ (which is an isomorphism since $\Sigma_{i,r}$ is a rational homology sphere) is given by $[\eta]\mapsto [\eta-d(\chi \theta)]$, where $\chi \in C^{\infty}(A_{i,r})$ is a cutoff function equal to 1 in a neighborhood of $\Sigma_{i,r}$, and which vanishes in a neighborhood of $\psi_i(\partial B_g(x,\frac{r}{2}))$. Here, we extend $\chi \theta$ to $\mathcal{B}_{i,r}$ by zero away from $\operatorname{supp}(\chi)$. Under the Poincar\'e-Lefschetz isomorphism $H^2(\mathcal{B}_{i,r},\partial \mathcal{B}_{i,r},\mathbb{R}) \cong H_2(\mathcal{B}_{i,r},\mathbb{R})$, the intersection pairing on $H_2(\mathcal{B}_{i,r},\mathbb{R})$ corresponds to the bilinear form $(\beta_1,\beta_2)\mapsto \int_{\mathcal{B}_{i,r}}\beta_1\wedge \beta_2$, and from this we can readily conclude the desired formula. By Lemma \ref{lem:intersectionnegative} and Stokes' theorem, we thus have
\begin{equation*}
    0 \geq \int_{\mathcal{B}_{i,r}} (\eta -d(\chi \theta)) \wedge (\eta-d(\chi \theta))=\int_{\mathcal{B}_{i,r}}\eta \wedge \eta - \int_{\Sigma_{i,r}} \theta \wedge \eta,
\end{equation*}
with equality if and only if $\eta-d(\chi \theta)$ is $d$-exact, or equivalently if $\eta$ is $d$-exact.
\end{proof}

By taking $\eta:=\widetilde{\omega}_{i,-1}|_{\mathcal{B}_{i,r}}$ and $\theta:=\alpha_{i,r}$ in Lemma \ref{lem:boundaryintegral}, we have 
\begin{equation*}
    0<\int_{\mathcal{B}_{i,r}}\widetilde{\omega}_{i,-1}^2 = q(c_1(\mathcal{B}_{i,r}),c_1(\mathcal{B}_{i,r}))+\int_{\Sigma_{i,r}} \alpha_{i,r}\wedge \widetilde{\omega}_{i,-1}.
\end{equation*}
On the other hand, Lemma \ref{lem:primitivealongsequence} implies
\begin{equation*}
\left| \int_{\partial \mathcal{B}_{i,r}} \alpha_{i,r}\wedge \widetilde{\omega}_{i,-1} \right| \leq Cr\mathcal{H}_{\widetilde{g}_{i,-1}}^3(\Sigma_{i,r}) \leq Cr^4,
\end{equation*}
hence combining expressions yields
\begin{equation*}
    -Cr^4 < q(c_1(\mathcal{B}_{i,r}),c_1(\mathcal{B}_{i,r})) \leq 0.
\end{equation*}
On the other hand, we may view $\frac{1}{2\pi}c_1(\mathcal{B}_{i,r})$ as a class in $H^2(\mathcal{B}_{i,r},\mathbb{Z})$, so that (by the long exact sequence of the pair) $\frac{|\Gamma|}{2\pi} c_1(\mathcal{B}_{i,r})$ has a representative in $H^2(\mathcal{B}_{i,r},\partial \mathcal{B}_{i,r},\mathbb{Z})$. The self-intersection of this representative is then
\begin{equation*}
    \frac{|\Gamma|^2}{(2\pi)^2} q(c_1(\mathcal{B}_{i,r}),c_1(\mathcal{B}_{i,r}))=q \left(\frac{|\Gamma|}{2\pi} c_1(\mathcal{B}_{i,r}),\frac{|\Gamma|}{2\pi}c_1(\mathcal{B}_{i,r})\right) \in \mathbb{Z}.
\end{equation*}
Choosing $r=r(\Gamma)>0$ sufficiently small, this implies $q(c_1(\mathcal{B}_{i,r}),c_1(\mathcal{B}_{i,r}))=0$, so by Lemma \ref{lem:boundaryintegral},
\begin{equation} \label{eq:c1vanishing}
[\widetilde{\omega}_{i,-1}|_{\mathcal{B}_{i,r}}]=c_1(\mathcal{B}_{i,r})=0 \qquad \text{ in } H^2(\mathcal{B}_{i,r},\mathbb{R}).
\end{equation} 
The adjunction formula \cite[Theorem 13.3.17(ii)]{McDuffSalamon} then rules out the existence of any symplectic $(-1)$-sphere, so that $(\mathcal{B}_{i,r},\widetilde{\omega}_{i,-1}|_{\mathcal{B}_{i,r}})$ is a minimal symplectic filling of $\Sigma_{i,r}$, where $\Sigma_{i,r}$ is contactomorphic to the standard contact structure on $\partial B_g(x,r)\cong \mathbb{S}^3/\Gamma$. 

\begin{Defn}
    A two-dimensional quotient singularity is said to be of class $T$, if it is a rational double point or a cyclic singularity of the form
    \begin{equation} \label{eq:TypeTsingularity}
        \frac{1}{dn^2}(1,dna-1), \qquad d\geq1,\quad n\geq 2, \quad \gcd(a,n)=1.
    \end{equation}
We say the singularity is a Wahl singularity if in addition $d=1$. 
\end{Defn}

\begin{Lemma} \label{lem:singclassT} If $B$ is a minimal symplectic filling of $\mathbb{S}^3/\Gamma$ with the standard contact structure such that $c_1(B)=0$ in $H^2(B,\mathbb{R})$, then $\mathbb{C}^2/\Gamma$ is a class $T$ singularity, and $B$ is symplectic deformation equivalent to a Milnor fiber of a $\mathbb{Q}$-Gorenstein smoothing of $\mathbb{C}^2/\Gamma$. 
\end{Lemma}
\begin{proof} By \cite[Proposition 2.12]{CPS}, $B$ is symplectic deformation equivalent to a Milnor fiber of $\mathbb{C}^2/\Gamma$. By \cite[3.9]{KollarSB}, this is a $\mathbb{Q}$-Gorenstein smoothing of $Y$, where $\rho:Y\to \mathbb{C}^2/\Gamma$ is a partial resolution such that $Y$ has only class $T$ singularities and $K_Y$ is $\rho$-ample. Letting $U$ denote the complement in $Y$ of small (disjoint) balls around the singular points, it follows that $U$ embeds in $B$, and the complex tangent bundles are homotopic, yielding 
\begin{equation} \label{eq:singclassT}
    c_1(K_Y)|_U=-c_1(B)|_U=0 \qquad \text{ in }  H^2(U,\mathbb{R}).
\end{equation}
Because the links of all singularities of $Y$ are $\mathbb{Q}$-homology spheres, the Mayer-Vietoris theorem implies that the map
\begin{equation*}
    H^2(Y,\mathbb{R})\to H^2(U,\mathbb{R})
\end{equation*}
induced by the inclusion $U\hookrightarrow Y$ is an isomorphism. By \eqref{eq:singclassT}, the image of $c_1(K_Y)$ under this map vanishes, hence $c_1(K_Y)=0$ in $H^2(Y,\mathbb{R})$. If there exists a curve $D \subseteq Y$ which is contracted by $\rho$, then because $K_Y$ is $\rho$-ample, we have 
\begin{equation*}
    0<K_Y \cdot D=\int_D c_1(K_Y)=0,
\end{equation*}
a contradiction. It follows that $\rho$ is an isomorphism, so that $\mathbb{C}^2/\Gamma$ is itself of class $T$. 
\end{proof}

\begin{Proposition} \label{prop:Wahl} Every singularity of $X$ is a Wahl singularity.
\end{Proposition}
\begin{proof}
By \eqref{eq:c1vanishing} and Lemma \ref{lem:singclassT} any singularity of $X$ is type $T$, and the corresponding minimal symplectic filling $\mathcal{B}_{i,r}$ is the corresponding Milnor fiber. If the singularity is a rational double point, then Lemma \ref{lem:intersectionnegative} gives $\Gamma = \{\pm 1\}$, so that $\mathcal{B}_{i,r}$ contains an integral homology class of self-intersection $-2$, contradicting the fact that no element of $H_2(\mathcal{V},\mathbb{Z})$ has self-intersection $-2$. Thus $\mathbb{C}^2/\Gamma$ is a cyclic singularity of the form \eqref{eq:TypeTsingularity}. By \cite[Section 2.2.2]{BeCh}, if $d>1$, then $\mathcal{B}_{i,r}$ contains an integral homology class of self-intersection $-2$ again, yielding a contradiction. Therefore $d=1$, so the claim follows.
\end{proof}

\subsection{$X$ is smooth}

\label{subsection:tangentflowisBCCD}

In this subsection, we complete the proof that any tangent flow of $(M,(\widetilde{g}_t)_{t\in [0,T)})$ is smooth. Given the previous results of this section, the proof is similar to that of \cite[Theorem 5.12]{CHM}.

\begin{Proposition} \label{prop:tangentflowsingularfiber}
    $X$ is the cylinder $\mathbb{C}\times \mathbb{P}^1$ or the BCCD shrinker.
\end{Proposition}
\begin{proof}
We first show $X$ is smooth. Suppose this is not the case, so that there are $k\geq 1$ Wahl singularities of $X$ by Proposition \ref{prop:Wahl}. Letting $F_1,...,F_N$ denote the irreducible components of the central fiber $(\pi \circ \rho)^{-1}(0) \subseteq \widetilde{X}$, it follows that $N\geq 2$. By Lemma \ref{lem:relativeMM}, the relative minimal model of $\pi \circ \rho$ is the projection map $\mathbb{C}\times \mathbb{P}^1 \to \mathbb{C}$, hence $\pi \circ \rho$ factors as a composition 
\begin{equation*}
\widetilde{X}=X_{N} \to X_{N-1} \to \cdots \to X_1 := \mathbb{C}\times \mathbb{P}^1 \to \mathbb{C},
\end{equation*}
where each map $X_j \to X_{j-1}$ is the blowup at a $X_{j-1}$ at a point on the central fiber of $X_{j-1}\to \mathbb{C}$. Each such blowup decreases the trace of the intersection matrix corresponding to the irreducible components of the central fiber by at most $3$, and the blowup $X_2 \to X_1$ decreases the trace by exactly $2$. It follows that the trace $\theta$ of $(F_i \cdot F_j)_{i,j=1}^N$ satisfies 
\begin{equation} \label{eq:trace1}
    \theta \geq -3(N-2)-2=-3N+4.
\end{equation}
On the other hand, the $F_i$ consist of exactly the proper transforms of the $Q$ irreducible components of $\pi^{-1}(0)$ along with the chains of $\mathbb{P}^1$ arising as the minimal resolutions of the Wahl singularities of $X$ (by Proposition \ref{prop:Wahl}). Let $n_i$ be the number of irreducible components of the exceptional set of the minimal resolution of the $i$th Wahl singularity. First suppose $Q=2$. The argument of \cite[Proof of Theorem 5.12]{CHM} then implies 
\begin{equation} \label{eq:trace2}
    \theta = \sum_{i=1}^k (-3(n_i-1)-4)-2,
\end{equation}
where we used that the proper transform of each component of $\pi^{-1}(0)$ is a $(-1)$-curve by Lemma \ref{lem:curvenumbers}. On the other hand, we have
\begin{equation} \label{eq:trace3}
    2+\sum_{i=1}^k n_i =N.
\end{equation}
Combining \eqref{eq:trace2} and \eqref{eq:trace3}, we therefore obtain
\begin{equation*}
    \theta = -3(N-2)-k-2 \leq -3N-k+4
\end{equation*}
Reconciling this with \eqref{eq:trace1} yields $k=0$, a contradiction. By Proposition \ref{prop:components}, we thus have $Q=1$. By arguing as above, using the fact that the proper transform of the unique component of $E$ has self intersection $0$ or $-1$ by Lemma \ref{lem:curvenumbers}, we obtain
\begin{equation} \label{eq:trace4} \theta \leq \sum_{i=1}^k (-3(n_i-1)-4), \qquad 1+\sum_{i=1}^k n_i =N,   
\end{equation}
so that $\theta \leq -3N+3-k$, again contradicting \eqref{eq:trace1}.

Therefore, $X$ is smooth. By the classification of \cite{CCD,BCCD}, it follows that $X$ is either the FIK shrinker on the total space of $\mathcal{O}_{\mathbb{P}^1}(-1)$, the cylinder $\mathbb{C}\times \mathbb{P}^1$, or the BCCD shrinker. By Proposition \ref{prop:curvaturenottozero}, $X$ is not the FIK shrinker, so the claim follows.
\end{proof}

\section{Proof of Theorem \ref{thm:main}}
\label{section:TypeI}

In this section, we complete the proof of Theorem \ref{thm:main}.

\begin{Lemma} \label{lem:compactnessoffinaltime} 
For any sequence $(x_i,t_i)\in M \times [0,T)$ with $t_i \nearrow T$, we can pass to a subsequence so that for some $\mu \in M_T$, the following hold:
\begin{enumerate}
    \item $\widetilde{K}(x_i,t_i;\cdot,\cdot)\to \widetilde{K}(\cdot,\cdot)$ in $C_{\operatorname{loc}}^{\infty}(M\times[0,T))$, where $d\mu_t=\widetilde{K}(\cdot,t)d\widetilde{g}_t$,

    \item $\lim_{i\to \infty} d_{W_1}^{\widetilde{g}_t}(\mu_t,\widetilde{\nu}_{x_i,t_i;t})=0$ for all $t\in [0,T)$. 

    \item $\lim_{i\to \infty} \mathcal{N}_{x_i,t_i}^{\widetilde{g}}(\tau) = \mathcal{N}_{\mu}^{\widetilde{g}}(\tau)$ for all $\tau \in (0,T)$. 
\end{enumerate}
\end{Lemma}
\begin{proof} This follows from \cite[Lemma 2.2]{mant} and \cite[Lemma 3.1]{HallLower}. 
\end{proof}

\begin{Lemma} \label{lem:comparisonofentropy} Letting $W_{\operatorname{BCCD}}, W_{\operatorname{Cyl}}$ be the Nash entropies of the BCCD shrinker and the round shrinking $\mathbb{C}\times \mathbb{P}^1$, respectively, we have
\begin{equation*}
    \log(\frac{1}{2}) < W_{\operatorname{BCCD}} < W_{\operatorname{Cyl}}.
\end{equation*}
\end{Lemma}
\begin{proof} By \cite{CHI}, $W_{\operatorname{Cyl}}=\log(2)-1$. By \cite[Example 2.33]{CCD} (see also \cite[Appendix A]{NO}), we have 
\begin{equation*}
    W_{\operatorname{BCCD}} = -2+\min_{x>0} \log \left( \frac{e^{2x}-e^x}{x^2}\right).
\end{equation*}
Jensen's inequality gives 
\begin{equation*}
    \frac{e^{2x}-e^x}{x^2} \geq \frac{e^{\frac{3x}{2}}}{x}.
\end{equation*}
An elementary computation then gives $\frac{e^{\frac{3x}{2}}}{x} \geq \frac{3}{2}e$ for all $x>0$, hence 
\begin{equation*}
    W_{\operatorname{BCCD}} \geq \log(\frac{3}{2}) -1 > \log(\frac{1}{2}).
\end{equation*}
Finally, we estimate
\begin{equation*}
    W_{\operatorname{BCCD}} \leq -2+\log \left( 4(e-e^{\frac{1}{2}}) \right) <\log(2)-1 = W_{\operatorname{Cyl}}.
\end{equation*}
\end{proof}

\begin{proof}[Proof of Theorem \ref{thm:main}] \eqref{thm:main:TypeI}
Suppose by way of contradiction that there exist $(x_i,t_i)\in M\times [0,T)$ such that $|\operatorname{Rm}_{\widetilde{g}_{t_i}}|_{\widetilde{g}_{t_i}}(x_i)(T-t_i) \to \infty$, and define
\begin{equation*}
    \widetilde{g}_{i,t}:=(T-t_i)^{-1}\widetilde{g}_{t_i+(T-t_i)t}, \qquad \widetilde{\nu}_t^i:= \widetilde{\nu}_{x_i,t_i;t_i+(T-t_i)t}
\end{equation*}
for $t\in [-(T-t_i)^{-1}t_i,0]$. By assumption, we then have $\lim_{i \to \infty} r_{\operatorname{Rm}}^{\widetilde{g}_i}(x_i,0)=0$, so by \cite[Theorem 10.2]{Bam1}, there exists $\epsilon_0>0$ such that for any $\tau>0$, we have
\begin{align} \label{eq:nontrivialentropy}
    \limsup_{i \to \infty} \mathcal{N}_{x_i,0}^{\widetilde{g}_i}(\tau) \leq -\epsilon_0.
\end{align}
After passing to a subsequence, we may assume that 
\begin{equation*}
    (M,(\widetilde{g}_{i,t})_{t\in [-(T-t_i)^{-1}t_i,0)},(\widetilde{\nu}_t^i)_{t\in [-(T-t_i)^{-1}t_i,0]})\xrightarrow[i\to \infty]{\mathbb{F}} (\mathcal{Z},(\mu_t)_{t\in (-\infty,0]})
\end{equation*}
for some $H_4$-concentrated metric flow $\mathcal{Z}$. Using \cite[Theorem 15.45]{Bam3} to pass \eqref{eq:nontrivialentropy} to the limit, we obtain
\begin{align} \label{eq:obviouslynontrivial}
    \mathcal{N}_{\mu}(\tau)\leq -\epsilon_0
\end{align}
for all $\tau>0$. Letting $\mathcal{Y}$ be any tangent flow of $\mathcal{Z}$ at $\mu$, it follows from \eqref{eq:obviouslynontrivial} that $\mathcal{Y}$ is modeled on some nontrivial orbifold shrinking gradient K\"ahler-Ricci soliton $(Y,g_Y)$. By Lemma \ref{lem:noncompact}, $Y$ is noncompact. If $Y$ were smooth, then it would be the BCCD shrinker, $\mathbb{C}\times \mathbb{P}^1$, or the FIK shrinker; in any case, it would contain a holomorphic curve of self-intersection $0$ or $-1$. Then a diagonal argument and \cite[Lemma 2.7]{BCCD} would yield a contradiction. Thus $Y$ has at least one orbifold point $y_{\ast}$. By a diagonal argument, there exist $Q_i >0$ and $(y_{i},s_{i})\in M\times[0,T)$ such that, setting $\widehat{g}_{i,s}:=Q_{i}\widetilde{g}_{s_{i}+Q_{i}^{-1}s}$ and $\widehat{\nu}_s^i:= \widetilde{\nu}_{y_i,s_i,s_i+Q_i^{-1}s}$ we have
\begin{equation*} (M,(\widehat{g}_{i,s})_{s\in [-Q_is_i,0)},(\widehat{\nu}_s^i)_{s\in [-Q_is_i,0)})\xrightarrow[i\to\infty]{\mathbb{F}}(\mathcal{Y},(\nu_{y_{\ast}})_{s\in(-\infty,0)}),\end{equation*}
and $Q_{i}(T-s_{i})\to\infty$. We know the unique tangent flow of $\mathcal{Y}$ based at $y_{\ast}$ is a static flow modeled on $C(\mathbb{S}^{3}/\Gamma)$ for some finite subgroup $\Gamma\leq U(2)$ acting freely on $\mathbb{S}^3$, so that \begin{equation*} \lim_{\tau\searrow0}\mathcal{N}_{y_{\ast}}(\tau)=\log \frac{1}{|\Gamma|}. \end{equation*}
For any $\sigma \in (0,T)$, we have
\begin{equation}\label{eq-entropy upper bound}
\mathcal{N}_{y_i,s_i}^{\widetilde{g}}(\sigma)=\mathcal{N}_{y_i,0}^{\widehat{g}_i}(Q_i \sigma)\leq \frac{3}{4}\log\frac{1}{2}+\frac{1}{4}W_{\operatorname{BCCD}}
\end{equation}
when $i=i(\sigma)\in \mathbb{N}$ is sufficiently large, where $s_i \to T$. After passing to a subsequence, Lemma \ref{lem:compactnessoffinaltime} gives $\widehat{\mu}\in M_T$ such that for all $t\in (0,T)$
\begin{equation}\label{entropy-convergence}
    \lim_{i \to \infty} \mathcal{N}_{y_i,s_i}^{\widetilde{g}}(s_i-t)=\mathcal{N}_{\widehat{\mu}}^{\widetilde{g}}(T-t).
\end{equation}
 By Proposition \ref{prop:tangentflowsingularfiber}, Lemma \ref{lem:comparisonofentropy}, and \cite[Theorem 2.37]{Bam3}, we know that 
\begin{equation}\label{entropy lower bound}
\lim_{\tau\searrow 0}\mathcal{N}_{\widehat{\mu}}^{\widetilde{g}}(\tau)\geq W_{\mathrm{BCCD}}>\log(\frac12).
\end{equation}
Choose $\tau_0>0$ sufficiently small and fix $0<\sigma<\tau_0$.
Then, for sufficiently large $i=i(\sigma,\tau_0)\in\mathbb{N}$,
\eqref{eq-entropy upper bound}, \eqref{entropy-convergence},
and \eqref{entropy lower bound} yield a contradiction.

\eqref{thm:main:regular},\eqref{thm:main:singular} Fix $p\in M$ and a sequence $t_i\nearrow T$. By Lemma \ref{lem:compactnessoffinaltime}, after passing to a subsequence, the conjugate heat kernels based at $(p,t_i)$ converge to some $\mu\in M_T$. The Type I bound, the Gaussian upper bound in \cite[Proposition 2.8]{mant}, and Bishop--Gromov volume comparison imply
\begin{equation*}
    d_{W_1}^{\widetilde{g}_t}(\delta_p,\mu_t)
    =\int_M d_{\widetilde{g}_t}(p,q)\,d\mu_t(q)
    \leq C\sqrt{T-t}.
\end{equation*}
If $(z_t,t)$ is an $H_4$-center of $\mu$, then
\begin{equation}\label{eq-distance bounded}
    d_{\widetilde{g}_t}(p,z_t)
    \leq d_{W_1}^{\widetilde{g}_t}(\delta_p,\mu_t)
    +d_{W_1}^{\widetilde{g}_t}(\mu_t,\delta_{z_t})
    \leq C\sqrt{T-t}.
\end{equation}
Combining the uniform Lipschitz bound for $\widetilde{\pi}$ from \cite[Lemma 7.3]{songnotes} with Lemma \ref{lem:Hncenters}\eqref{lem:Hncenters1}, we obtain
\begin{equation*}
    d_{g_Y}\bigl(\widetilde{\pi}(p),\widehat{\pi}(\mu)\bigr)
    \leq C d_{\widetilde{g}_t}(p,z_t)+C\sqrt{T-t}
    \leq C\sqrt{T-t}.
\end{equation*}
Letting $t\nearrow T$ gives $\widehat{\pi}(\mu)=\widetilde{\pi}(p)$.

Applying Theorem \ref{bamconvergence} and Proposition \ref{prop:tangentflowregularfiber} and passing to a further subsequence, we obtain a tangent flow at $\mu$ modeled on a smooth K\"ahler--Ricci shrinker $(X,J,g)$.  By \cite[Proposition 2.7]{mant} and Perelman's no-local-collapsing theorem, we can apply \cite[Proposition 2.7]{HallLower}. Together with \eqref{eq-distance bounded}, this yields a point \(p_{\infty}\in X\) such that
\begin{equation*}
    (M,(T-t_i)^{-\frac{1}{2}}d_{\widetilde{g}_{t_i}},p) \to (X,d_g,p_{\infty})
\end{equation*}
subsequentially in the pointed Gromov-Hausdorff sense as $i\to \infty$. By \eqref{thm:main:TypeI} and Hamilton's compactness theorem, the convergence occurs in the smooth pointed Cheeger-Gromov sense. 

If $p\in \widetilde{U}$, then $X$ is the round cylinder $\mathbb{C}\times \mathbb{P}^1$ by Proposition \ref{prop:tangentflowregularfiber}, hence \eqref{thm:main:regular} follows. Suppose instead that $p \in M\setminus \widetilde{U}$, so that by Proposition \ref{prop:tangentflowsingularfiber}, it remains to show that $X$ is not the round cylinder $\mathbb{C}\times \mathbb{P}^1$. By the smooth Cheeger-Gromov convergence, there exists a precompact open exhaustion $(U_i)$ of $X$ along with open embeddings $\psi_i:U_i \to M$ such that
\begin{equation*}\widetilde{g}_{i,-1}:=(T-t_i)^{-1}
\psi_i^{\ast}\widetilde{g}_{t_i}\to g,\qquad \psi_i^{\ast}\widetilde{J}\to J, \qquad \psi_i^{-1}(p)\to p_{\infty}
\end{equation*}
as $i\to \infty$. Let $D\subseteq M$ be a $(-1)$-curve passing through $p$.
Using the monotonicity of minimal surfaces, Type I curvature bound, Perelman’s no local collapsing and arguing as in \cite[Lemma 6.6]{LZ26}, we know that  the extrinsic $\widetilde{g}_{i,-1}$-diameter of $D$ is bounded uniformly in $i\in \mathbb{N}$; see also \cite[Proof of Proposition 5.2]{JScollapsing}. It follows that $D \Subset \psi_i(U_i)$ for all sufficiently large $i\in \mathbb{N}$, so that $\psi_i^{-1}(D)$ represents a class in $H_2(X,\mathbb{Z})$ with strictly negative self-intersection. Because the intersection form on $H_2(\mathbb{C}\times \mathbb{P}^1,\mathbb{Z})$ is zero, this implies $X$ is the BCCD shrinker.
\end{proof}

\printbibliography

@article{LZ26,
  author = {Li, Yu and Zhang, Junsheng},
  title = {Gromov--Hausdorff Limits of Noncollapsed K{\"a}hler--Ricci Flows and the Geometry of Ricci Shrinkers},
  journal = {arXiv preprint arXiv:2607.25644},
  year = {2026}
}

@book{BKV,
  title={{Compact complex surfaces}},
  author={Barth, Wolf and Hulek, Klaus and Peters, Chris and Van de Ven, Antonius},
  volume={4},
  year={2003},
  publisher={Springer Science \& Business Media}
}

@article{HNP,
  title={{Symplectic convexity theorems and coadjoint orbits}},
  author={Hilgert, Joachim and Neeb, Karl-Hermann and Plank, Werner},
  journal={Compositio Mathematica},
  volume={94},
  number={2},
  pages={129--180},
  year={1994}
}

@misc{Bam1,
  title = {Entropy and Heat Kernel Bounds on a {{Ricci}} Flow Background},
  author = {Bamler, Richard H.},
  year = {2021},
  month = sep,
  number = {arXiv:2008.07093},
  eprint = {2008.07093},
  eprinttype = {arxiv},
  primaryclass = {math},
  publisher = {{arXiv}},
  archiveprefix = {arXiv},
  langid = {english}
}

@misc{Bam2,
  title = {Compactness Theory of the Space of {{Super Ricci}} Flows},
  author = {Bamler, Richard H.},
  year = {2021},
  month = jun,
  number = {arXiv:2008.09298},
  eprint = {2008.09298},
  eprinttype = {arxiv},
  primaryclass = {math},
  publisher = {{arXiv}},
  archiveprefix = {arXiv},
  langid = {english}
}

@misc{Bam3,
  title = {Structure Theory of Non-Collapsed Limits of {{Ricci}} Flows},
  author = {Bamler, Richard H.},
  year = {2021},
  month = sep,
  number = {arXiv:2009.03243},
  eprint = {2009.03243},
  eprinttype = {arxiv},
  primaryclass = {math},
  publisher = {{arXiv}},
  archiveprefix = {arXiv},
  langid = {english}
}

@incollection {BandoMabuchi,
    AUTHOR = {Bando, Shigetoshi and Mabuchi, Toshiki},
     TITLE = {Uniqueness of {E}instein {K}\"ahler metrics modulo connected
              group actions},
 BOOKTITLE = {Algebraic geometry, {S}endai, 1985},
    SERIES = {Adv. Stud. Pure Math.},
    VOLUME = {10},
     PAGES = {11--40},
 PUBLISHER = {North-Holland, Amsterdam},
      YEAR = {1987},
      ISBN = {0-444-70313-6},
   MRCLASS = {53C25 (32C10 32J99 32M05 53C55 58E20)},
  MRNUMBER = {946233},
MRREVIEWER = {Yusuke\ Sakane},
       DOI = {10.2969/aspm/01010011},
       URL = {https://doi.org/10.2969/aspm/01010011},
}

@article {BCCD,
    AUTHOR = {Bamler, Richard H. and Cifarelli, Charles and Conlon, Ronan J.
              and Deruelle, Alix},
     TITLE = {A new complete two-dimensional shrinking gradient
              {K}\"ahler-{R}icci soliton},
   JOURNAL = {Geom. Funct. Anal.},
  FJOURNAL = {Geometric and Functional Analysis},
    VOLUME = {34},
      YEAR = {2024},
    NUMBER = {2},
     PAGES = {377--392},
      ISSN = {1016-443X,1420-8970},
   MRCLASS = {53C55 (32Q15 53C25)},
  MRNUMBER = {4715366},
MRREVIEWER = {Tam\'as\ Darvas},
       DOI = {10.1007/s00039-024-00668-9},
       URL = {https://doi-org.proxy.libraries.rutgers.edu/10.1007/s00039-024-00668-9},
}

@article {BeCh,
    AUTHOR = {Behnke, Kurt and Christophersen, Jan Arthur},
     TITLE = {{$M$}-resolutions and deformations of quotient singularities},
   JOURNAL = {Amer. J. Math.},
  FJOURNAL = {American Journal of Mathematics},
    VOLUME = {116},
      YEAR = {1994},
    NUMBER = {4},
     PAGES = {881--903},
      ISSN = {0002-9327,1080-6377},
   MRCLASS = {14B07 (14J17 32S30)},
  MRNUMBER = {1287942},
MRREVIEWER = {Jonathan\ M.\ Wahl},
       DOI = {10.2307/2375004},
       URL = {https://doi.org/10.2307/2375004},
}

@book {BG,
    AUTHOR = {Boyer, Charles P. and Galicki, Krzysztof},
     TITLE = {Sasakian geometry},
    SERIES = {Oxford Mathematical Monographs},
 PUBLISHER = {Oxford University Press, Oxford},
      YEAR = {2008},
     PAGES = {xii+613},
      ISBN = {978-0-19-856495-9},
   MRCLASS = {53C25 (14J45 32J27 53-02 57R30 57S25)},
  MRNUMBER = {2382957},
MRREVIEWER = {Andrew\ Swann},
}

@article {Caosoliton,
    AUTHOR = {Cao, Huai-Dong},
     TITLE = {Limits of solutions to the {K}\"ahler-{R}icci flow},
   JOURNAL = {J. Differential Geom.},
  FJOURNAL = {Journal of Differential Geometry},
    VOLUME = {45},
      YEAR = {1997},
    NUMBER = {2},
     PAGES = {257--272},
      ISSN = {0022-040X,1945-743X},
   MRCLASS = {53C21 (53C55)},
  MRNUMBER = {1449972},
MRREVIEWER = {Ben\ Andrews},
       URL = {http://projecteuclid.org/euclid.jdg/1214459797},
}

@article {CarrKutt,
    AUTHOR = {Carrell, James B. and Kuttler, Jochen},
     TITLE = {Smooth points of {$T$}-stable varieties in {$G/B$} and the
              {P}eterson map},
   JOURNAL = {Invent. Math.},
  FJOURNAL = {Inventiones Mathematicae},
    VOLUME = {151},
      YEAR = {2003},
    NUMBER = {2},
     PAGES = {353--379},
      ISSN = {0020-9910,1432-1297},
   MRCLASS = {14M17},
  MRNUMBER = {1953262},
MRREVIEWER = {Raika\ Dehy},
       DOI = {10.1007/s00222-002-0256-5},
       URL = {https://doi.org/10.1007/s00222-002-0256-5},
}

@article{CCD,
    author = {Cifarelli, Charles and Conlon, Ronan and Deruelle, Alix},
    title = {Type I singularities},
    journal = {J. Eur. Math. Soc},
    year = {2024}
}

@article {ChenWangsurfaces,
    AUTHOR = {Chen, Xiuxiong and Wang, Bing},
     TITLE = {The {K}\"ahler {R}icci flow on {F}ano surfaces ({I})},
   JOURNAL = {Math. Z.},
  FJOURNAL = {Mathematische Zeitschrift},
    VOLUME = {270},
      YEAR = {2012},
    NUMBER = {1-2},
     PAGES = {577--587},
      ISSN = {0025-5874,1432-1823},
   MRCLASS = {53C44 (32G05 32Q20)},
  MRNUMBER = {2875850},
MRREVIEWER = {Julien\ Keller},
       DOI = {10.1007/s00209-010-0813-3},
       URL = {https://doi.org/10.1007/s00209-010-0813-3},
}

@misc{CHI,
      title={Gaussian densities and stability for some Ricci solitons}, 
      author={Huai-Dong Cao and Richard S. Hamilton and Tom Ilmanen},
      year={2004},
      eprint={math/0404165},
      archivePrefix={arXiv},
      primaryClass={math.DG},
      url={https://arxiv.org/abs/math/0404165}, 
}

@misc{CHL,
      title={K\"ahler-Ricci Tangent Flows in the Analytic Minimal Model Program}, 
      author={Longteng Chen and Max Hallgren and Lucas Lavoyer},
      year={2026},
      eprint={2608.19152},
      archivePrefix={arXiv},
      primaryClass={math.DG},
      url={https://arxiv.org/abs/2608.19152}, 
}

@misc{CHM,
      title={Non-collapsed finite time singularities of the Ricci flow on compact K\"ahler surfaces are of Type I}, 
      author={Ronan J. Conlon and Max Hallgren and Zilu Ma},
      year={2025},
      eprint={2502.19804},
      archivePrefix={arXiv},
      primaryClass={math.DG},
      url={https://arxiv.org/abs/2502.19804}, 
}

@article {CPS,
    AUTHOR = {Choi, Hakho and Park, Heesang and Shin, Dongsoo},
     TITLE = {Symplectic fillings of quotient surface singularities and
              minimal model program},
   JOURNAL = {J. Korean Math. Soc.},
  FJOURNAL = {Journal of the Korean Mathematical Society},
    VOLUME = {58},
      YEAR = {2021},
    NUMBER = {2},
     PAGES = {419--437},
      ISSN = {0304-9914,2234-3008},
   MRCLASS = {57R40 (14B07 14E30 57R55)},
  MRNUMBER = {4221573},
MRREVIEWER = {Burak\ Ozbagci},
       DOI = {10.4134/JKMS.j200093},
       URL = {https://doi.org/10.4134/JKMS.j200093},
}

@article {crisp,
    AUTHOR = {Crisp, J. S. and Hillman, J. A.},
     TITLE = {Embedding {S}eifert fibred {$3$}-manifolds and {${\rm
              Sol}^3$}-manifolds in {$4$}-space},
   JOURNAL = {Proc. London Math. Soc. (3)},
  FJOURNAL = {Proceedings of the London Mathematical Society. Third Series},
    VOLUME = {76},
      YEAR = {1998},
    NUMBER = {3},
     PAGES = {685--710},
      ISSN = {0024-6115,1460-244X},
   MRCLASS = {57N35 (57N10 57N13)},
  MRNUMBER = {1620508},
MRREVIEWER = {Charles\ Livingston},
       DOI = {10.1112/S0024611598000379},
       URL = {https://doi-org.proxy.libraries.rutgers.edu/10.1112/S0024611598000379},
}

@misc{FH,
      title={An $\epsilon$-Regularity Theorem for Non-collapsed Ricci Flow}, 
      author={Harry Fluck and Max Hallgren},
      year={2025},
      eprint={2509.14154},
      archivePrefix={arXiv},
      primaryClass={math.DG},
      url={https://arxiv.org/abs/2509.14154}, 
}

@article {FIK,
    AUTHOR = {Feldman, Mikhail and Ilmanen, Tom and Knopf, Dan},
     TITLE = {Rotationally symmetric shrinking and expanding gradient
              {K}\"ahler-{R}icci solitons},
   JOURNAL = {J. Differential Geom.},
  FJOURNAL = {Journal of Differential Geometry},
    VOLUME = {65},
      YEAR = {2003},
    NUMBER = {2},
     PAGES = {169--209},
      ISSN = {0022-040X,1945-743X},
   MRCLASS = {53C44 (32Q20 53C25)},
  MRNUMBER = {2058261},
MRREVIEWER = {Xi\ Ping\ Zhu},
       URL = {http://projecteuclid.org/euclid.jdg/1090511686},
}

@article {Fong,
    AUTHOR = {Fong, Frederick Tsz-Ho},
     TITLE = {K\"ahler-{R}icci flow on projective bundles over
              {K}\"ahler-{E}instein manifolds},
   JOURNAL = {Trans. Amer. Math. Soc.},
  FJOURNAL = {Transactions of the American Mathematical Society},
    VOLUME = {366},
      YEAR = {2014},
    NUMBER = {2},
     PAGES = {563--589},
      ISSN = {0002-9947,1088-6850},
   MRCLASS = {53C44 (53C55 55R25)},
  MRNUMBER = {3130308},
MRREVIEWER = {Julien\ Keller},
       DOI = {10.1090/S0002-9947-2013-05726-1},
       URL = {https://doi-org.proxy.libraries.rutgers.edu/10.1090/S0002-9947-2013-05726-1},
}

@book {Gaborbook,
    AUTHOR = {Sz\'ekelyhidi, G\'abor},
     TITLE = {An introduction to extremal {K}\"ahler metrics},
    SERIES = {Graduate Studies in Mathematics},
    VOLUME = {152},
 PUBLISHER = {American Mathematical Society, Providence, RI},
      YEAR = {2014},
     PAGES = {xvi+192},
      ISBN = {978-1-4704-1047-6},
   MRCLASS = {53C55 (14L24 32Q20 53C25)},
  MRNUMBER = {3186384},
MRREVIEWER = {Andrew\ Bucki},
       DOI = {10.1090/gsm/152},
       URL = {https://doi.org/10.1090/gsm/152},
}

@article {HallLower,
    AUTHOR = {Hallgren, Max},
     TITLE = {Ricci flow with {R}icci curvature and volume bounded below},
   JOURNAL = {Math. Ann.},
  FJOURNAL = {Mathematische Annalen},
    VOLUME = {390},
      YEAR = {2024},
    NUMBER = {2},
     PAGES = {2633--2706},
      ISSN = {0025-5831,1432-1807},
   MRCLASS = {53E20 (53C21)},
  MRNUMBER = {4801837},
       DOI = {10.1007/s00208-024-02821-z},
       URL = {https://doi-org.proxy.libraries.rutgers.edu/10.1007/s00208-024-02821-z},
}

@misc{HZ,
      title={Singular K\"ahler--Ricci shrinkers and polarized Fano fibrations}, 
      author={Max Hallgren and Junsheng Zhang},
      year={2026},
      eprint={2605.25213},
      archivePrefix={arXiv},
      primaryClass={math.DG},
      url={https://arxiv.org/abs/2605.25213}, 
}

@misc{JScollapsing,
      title={Finite-Time Singularities of the K\"ahler--Ricci Flow on Fano Bundles}, 
      author={Wangjian Jian and Jian Song},
      year={2026},
      eprint={2609.02878},
      archivePrefix={arXiv},
      primaryClass={math.DG},
      url={https://arxiv.org/abs/2609.02878}, 
}

@incollection {Koiso,
    AUTHOR = {Koiso, Norihito},
     TITLE = {On rotationally symmetric {H}amilton's equation for
              {K}\"ahler-{E}instein metrics},
 BOOKTITLE = {Recent topics in differential and analytic geometry},
    SERIES = {Adv. Stud. Pure Math.},
    VOLUME = {18-{\rm I}},
     PAGES = {327--337},
 PUBLISHER = {Academic Press, Boston, MA},
      YEAR = {1990},
      ISBN = {0-12-001018-6},
   MRCLASS = {53C25 (32L07 53C55 58G30)},
  MRNUMBER = {1145263},
MRREVIEWER = {Hajime\ Tsuji},
       DOI = {10.2969/aspm/01810327},
       URL = {https://doi.org/10.2969/aspm/01810327},
}

@article {KollarSB,
    AUTHOR = {Koll\'ar, J. and Shepherd-Barron, N. I.},
     TITLE = {Threefolds and deformations of surface singularities},
   JOURNAL = {Invent. Math.},
  FJOURNAL = {Inventiones Mathematicae},
    VOLUME = {91},
      YEAR = {1988},
    NUMBER = {2},
     PAGES = {299--338},
      ISSN = {0020-9910,1432-1297},
   MRCLASS = {14J10 (14D20 14J30 32G10 32G13)},
  MRNUMBER = {922803},
MRREVIEWER = {Yujiro\ Kawamata},
       DOI = {10.1007/BF01389370},
       URL = {https://doi.org/10.1007/BF01389370},
}

@article {Licollapsing,
    AUTHOR = {Li, Jiangtao},
     TITLE = {On the tangent flow to the collapsing {K}\"ahler-{R}icci flow
              on {H}irzebruch surfaces},
   JOURNAL = {Internat. J. Math.},
  FJOURNAL = {International Journal of Mathematics},
    VOLUME = {36},
      YEAR = {2025},
    NUMBER = {12},
     PAGES = {Paper No. 2550048, 12},
      ISSN = {0129-167X,1793-6519},
   MRCLASS = {53E30 (53C25)},
  MRNUMBER = {4966651},
MRREVIEWER = {Yang\ Li},
       DOI = {10.1142/S0129167X2550048X},
       URL = {https://doi.org/10.1142/S0129167X2550048X},
}

@article {LuPseudo,
    AUTHOR = {Lu, Peng},
     TITLE = {A local curvature bound in {R}icci flow},
   JOURNAL = {Geom. Topol.},
  FJOURNAL = {Geometry \& Topology},
    VOLUME = {14},
      YEAR = {2010},
    NUMBER = {2},
     PAGES = {1095--1110},
      ISSN = {1465-3060,1364-0380},
   MRCLASS = {53C44},
  MRNUMBER = {2629901},
MRREVIEWER = {Esther\ Cabezas Rivas},
       DOI = {10.2140/gt.2010.14.1095},
       URL = {https://doi-org.proxy.libraries.rutgers.edu/10.2140/gt.2010.14.1095},
}

@article {LiWang,
    AUTHOR = {Li, Yu and Wang, Bing},
     TITLE = {On {K}\"ahler {R}icci shrinker surfaces},
   JOURNAL = {Acta Math.},
  FJOURNAL = {Acta Mathematica},
    VOLUME = {236},
      YEAR = {2026},
    NUMBER = {1},
     PAGES = {1--50},
      ISSN = {0001-5962,1871-2509},
   MRCLASS = {53C25 (53C55)},
  MRNUMBER = {5055602},
       DOI = {10.4310/acta.2026.v236.n1.a1},
       URL = {https://doi.org/10.4310/acta.2026.v236.n1.a1},
}

@article {mant,
    AUTHOR = {Mantegazza, Carlo and M\"uller, Reto},
     TITLE = {Perelman's entropy functional at {T}ype {I} singularities of
              the {R}icci flow},
   JOURNAL = {J. Reine Angew. Math.},
  FJOURNAL = {Journal f\"ur die Reine und Angewandte Mathematik. [Crelle's
              Journal]},
    VOLUME = {703},
      YEAR = {2015},
     PAGES = {173--199},
      ISSN = {0075-4102,1435-5345},
   MRCLASS = {53C44},
  MRNUMBER = {3353546},
MRREVIEWER = {Anqiang\ Zhu},
       DOI = {10.1515/crelle-2013-0039},
       URL = {https://doi-org.proxy.libraries.rutgers.edu/10.1515/crelle-2013-0039},
}

@book {matsuki,
    AUTHOR = {Matsuki, Kenji},
     TITLE = {Introduction to the {M}ori program},
    SERIES = {Universitext},
 PUBLISHER = {Springer-Verlag, New York},
      YEAR = {2002},
     PAGES = {xxiv+478},
      ISBN = {0-387-98465-8},
   MRCLASS = {14E30 (14-02 14E05 14J17 14J30)},
  MRNUMBER = {1875410},
MRREVIEWER = {Massimiliano\ Mella},
       DOI = {10.1007/978-1-4757-5602-9},
       URL = {https://doi.org/10.1007/978-1-4757-5602-9},
}

@book {McDuffSalamon,
    AUTHOR = {McDuff, Dusa and Salamon, Dietmar},
     TITLE = {Introduction to symplectic topology},
    SERIES = {Oxford Graduate Texts in Mathematics},
   EDITION = {Third},
 PUBLISHER = {Oxford University Press, Oxford},
      YEAR = {2017},
     PAGES = {xi+623},
      ISBN = {978-0-19-879490-5; 978-0-19-879489-9},
   MRCLASS = {53D35 (53D40 57R17 57R57 57R58)},
  MRNUMBER = {3674984},
MRREVIEWER = {Hansj\"org\ Geiges},
       DOI = {10.1093/oso/9780198794899.001.0001},
       URL = {https://doi.org/10.1093/oso/9780198794899.001.0001},
}

@article {mumford,
    AUTHOR = {Mumford, David},
     TITLE = {The topology of normal singularities of an algebraic surface
              and a criterion for simplicity},
   JOURNAL = {Inst. Hautes \'Etudes Sci. Publ. Math.},
  FJOURNAL = {Institut des Hautes \'Etudes Scientifiques. Publications
              Math\'ematiques},
    NUMBER = {9},
      YEAR = {1961},
     PAGES = {5--22},
      ISSN = {0073-8301,1618-1913},
   MRCLASS = {14.18 (14.55)},
  MRNUMBER = {153682},
MRREVIEWER = {T.\ Matsusaka},
       URL = {http://www.numdam.org/item?id=PMIHES_1961__9__5_0},
}

@misc{NO,
      title={Linear stability and instability of K\"ahler Ricci solitons}, 
      author={Keaton Naff and Tristan Ozuch},
      year={2025},
      eprint={2511.15885},
      archivePrefix={arXiv},
      primaryClass={math.DG},
      url={https://arxiv.org/abs/2511.15885}, 
}

@article {OO,
    AUTHOR = {Ohta, Hiroshi and Ono, Kaoru},
     TITLE = {Simple singularities and symplectic fillings},
   JOURNAL = {J. Differential Geom.},
  FJOURNAL = {Journal of Differential Geometry},
    VOLUME = {69},
      YEAR = {2005},
    NUMBER = {1},
     PAGES = {1--42},
      ISSN = {0022-040X,1945-743X},
   MRCLASS = {53D35 (32S45 53D45 57R17)},
  MRNUMBER = {2169581},
MRREVIEWER = {Paolo\ Lisca},
       DOI = {10.4310/jdg/1121540338},
       URL = {https://doi-org.proxy.libraries.rutgers.edu/10.4310/jdg/1121540338},
}

@book {SchwarzBook,
    AUTHOR = {Schwarz, G\"unter},
     TITLE = {Hodge decomposition---a method for solving boundary value
              problems},
    SERIES = {Lecture Notes in Mathematics},
    VOLUME = {1607},
 PUBLISHER = {Springer-Verlag, Berlin},
      YEAR = {1995},
     PAGES = {viii+155},
      ISBN = {3-540-60016-7},
   MRCLASS = {58G20 (35J99 35N10 47N20 58A14)},
  MRNUMBER = {1367287},
MRREVIEWER = {Robert\ McOwen},
       DOI = {10.1007/BFb0095978},
       URL = {https://doi.org/10.1007/BFb0095978},
}

@article {SesumTian,
    AUTHOR = {Sesum, Natasa and Tian, Gang},
     TITLE = {Bounding scalar curvature and diameter along the {K}\"ahler
              {R}icci flow (after {P}erelman)},
   JOURNAL = {J. Inst. Math. Jussieu},
  FJOURNAL = {Journal of the Institute of Mathematics of Jussieu. JIMJ.
              Journal de l'Institut de Math\'ematiques de Jussieu},
    VOLUME = {7},
      YEAR = {2008},
    NUMBER = {3},
     PAGES = {575--587},
      ISSN = {1474-7480,1475-3030},
   MRCLASS = {53C44 (53C55)},
  MRNUMBER = {2427424},
MRREVIEWER = {Julien\ Keller},
       DOI = {10.1017/S1474748008000133},
       URL = {https://doi.org/10.1017/S1474748008000133},
}

@misc{SunZhang,
      title={K\"ahler-Ricci shrinkers and Fano fibrations}, 
      author={Song Sun and Junsheng Zhang},
      year={2025},
      eprint={2410.09661},
      archivePrefix={arXiv},
      primaryClass={math.DG},
      url={https://arxiv.org/abs/2410.09661}, 
}

@article {SSW,
    AUTHOR = {Song, Jian and Sz\'ekelyhidi, G\'abor and Weinkove, Ben},
     TITLE = {The {K}\"ahler-{R}icci flow on projective bundles},
   JOURNAL = {Int. Math. Res. Not. IMRN},
  FJOURNAL = {International Mathematics Research Notices. IMRN},
      YEAR = {2013},
    NUMBER = {2},
     PAGES = {243--257},
      ISSN = {1073-7928,1687-0247},
   MRCLASS = {53C44 (57R22)},
  MRNUMBER = {3010688},
MRREVIEWER = {Bang\ Xiao},
       DOI = {10.1093/imrn/rnr265},
       URL = {https://doi-org.proxy.libraries.rutgers.edu/10.1093/imrn/rnr265},
}

@article {SWHirz,
    AUTHOR = {Song, Jian and Weinkove, Ben},
     TITLE = {The {K}\"ahler-{R}icci flow on {H}irzebruch surfaces},
   JOURNAL = {J. Reine Angew. Math.},
  FJOURNAL = {Journal f\"ur die Reine und Angewandte Mathematik. [Crelle's
              Journal]},
    VOLUME = {659},
      YEAR = {2011},
     PAGES = {141--168},
      ISSN = {0075-4102,1435-5345},
   MRCLASS = {53C44 (32Q15 53C23)},
  MRNUMBER = {2837013},
MRREVIEWER = {Kai\ Zheng},
       DOI = {10.1515/CRELLE.2011.071},
       URL = {https://doi-org.proxy.libraries.rutgers.edu/10.1515/CRELLE.2011.071},
}

@article {SongExt,
    AUTHOR = {Song, Jian},
     TITLE = {Finite-time extinction of the {K}\"ahler-{R}icci flow},
   JOURNAL = {Math. Res. Lett.},
  FJOURNAL = {Mathematical Research Letters},
    VOLUME = {21},
      YEAR = {2014},
    NUMBER = {6},
     PAGES = {1435--1449},
      ISSN = {1073-2780,1945-001X},
   MRCLASS = {53C44 (53C55)},
  MRNUMBER = {3335855},
MRREVIEWER = {Casey\ Lynn\ Kelleher},
       DOI = {10.4310/MRL.2014.v21.n6.a12},
       URL = {https://doi-org.proxy.libraries.rutgers.edu/10.4310/MRL.2014.v21.n6.a12},
}

@incollection {songnotes,
    AUTHOR = {Song, Jian and Weinkove, Ben},
     TITLE = {An introduction to the {K}\"ahler-{R}icci flow},
 BOOKTITLE = {An introduction to the {K}\"ahler-{R}icci flow},
    SERIES = {Lecture Notes in Math.},
    VOLUME = {2086},
     PAGES = {89--188},
 PUBLISHER = {Springer, Cham},
      YEAR = {2013},
      ISBN = {978-3-319-00818-9; 978-3-319-00819-6},
   MRCLASS = {53C44 (32Q15 53C55)},
  MRNUMBER = {3185333},
       DOI = {10.1007/978-3-319-00819-6\_3},
       URL = {https://doi-org.proxy.libraries.rutgers.edu/10.1007/978-3-319-00819-6_3},
}

@article {SW1,
    AUTHOR = {Song, Jian and Weinkove, Ben},
     TITLE = {Contracting exceptional divisors by the {K}\"ahler-{R}icci
              flow},
   JOURNAL = {Duke Math. J.},
  FJOURNAL = {Duke Mathematical Journal},
    VOLUME = {162},
      YEAR = {2013},
    NUMBER = {2},
     PAGES = {367--415},
      ISSN = {0012-7094,1547-7398},
   MRCLASS = {53C44 (14E05 32Q20)},
  MRNUMBER = {3018957},
MRREVIEWER = {Julien\ Keller},
       DOI = {10.1215/00127094-1962881},
       URL = {https://doi.org/10.1215/00127094-1962881},
}

@article {tiancalabi,
    AUTHOR = {Tian, G.},
     TITLE = {On {C}alabi's conjecture for complex surfaces with positive
              first {C}hern class},
   JOURNAL = {Invent. Math.},
  FJOURNAL = {Inventiones Mathematicae},
    VOLUME = {101},
      YEAR = {1990},
    NUMBER = {1},
     PAGES = {101--172},
      ISSN = {0020-9910,1432-1297},
   MRCLASS = {32L07 (32F07 53C25 53C55)},
  MRNUMBER = {1055713},
MRREVIEWER = {M.\ Kalka},
       DOI = {10.1007/BF01231499},
       URL = {https://doi.org/10.1007/BF01231499},
}

@article {TianZhu02,
    AUTHOR = {Tian, Gang and Zhu, Xiaohua},
     TITLE = {A new holomorphic invariant and uniqueness of
              {K}\"ahler-{R}icci solitons},
   JOURNAL = {Comment. Math. Helv.},
  FJOURNAL = {Commentarii Mathematici Helvetici},
    VOLUME = {77},
      YEAR = {2002},
    NUMBER = {2},
     PAGES = {297--325},
      ISSN = {0010-2571,1420-8946},
   MRCLASS = {32Q20 (53C25 53C55)},
  MRNUMBER = {1915043},
MRREVIEWER = {Peng\ Lu},
       DOI = {10.1007/s00014-002-8341-3},
       URL = {https://doi.org/10.1007/s00014-002-8341-3},
}

@article {TianZhu07,
    AUTHOR = {Tian, Gang and Zhu, Xiaohua},
     TITLE = {Convergence of {K}\"ahler-{R}icci flow},
   JOURNAL = {J. Amer. Math. Soc.},
  FJOURNAL = {Journal of the American Mathematical Society},
    VOLUME = {20},
      YEAR = {2007},
    NUMBER = {3},
     PAGES = {675--699},
      ISSN = {0894-0347,1088-6834},
   MRCLASS = {53C44 (53C55)},
  MRNUMBER = {2291916},
MRREVIEWER = {Julien\ Keller},
       DOI = {10.1090/S0894-0347-06-00552-2},
       URL = {https://doi.org/10.1090/S0894-0347-06-00552-2},
}

@article {TZ,
    AUTHOR = {Tosatti, Valentino and Zhang, Yuguang},
     TITLE = {Finite time collapsing of the {K}\"ahler-{R}icci flow on
              threefolds},
   JOURNAL = {Ann. Sc. Norm. Super. Pisa Cl. Sci. (5)},
  FJOURNAL = {Annali della Scuola Normale Superiore di Pisa. Classe di
              Scienze. Serie V},
    VOLUME = {18},
      YEAR = {2018},
    NUMBER = {1},
     PAGES = {105--118},
      ISSN = {0391-173X,2036-2145},
   MRCLASS = {53C44 (32Q15 53C55)},
  MRNUMBER = {3783785},
MRREVIEWER = {Frederick\ Tsz-Ho\ Fong},
}

@article {WangZhu,
    AUTHOR = {Wang, Xu-Jia and Zhu, Xiaohua},
     TITLE = {K\"ahler-{R}icci solitons on toric manifolds with positive
              first {C}hern class},
   JOURNAL = {Adv. Math.},
  FJOURNAL = {Advances in Mathematics},
    VOLUME = {188},
      YEAR = {2004},
    NUMBER = {1},
     PAGES = {87--103},
      ISSN = {0001-8708,1090-2082},
   MRCLASS = {53C25 (32Q20)},
  MRNUMBER = {2084775},
MRREVIEWER = {Christina\ W.\ T\o nnesen-Friedman},
       DOI = {10.1016/j.aim.2003.09.009},
       URL = {https://doi.org/10.1016/j.aim.2003.09.009},
}

@misc{XuZhang,
      title={Finite Time Singularities of Collapsing K\"ahler Ricci Flow on Ruled Surfaces}, 
      author={Tongxin Xu and Zhenlei Zhang},
      year={2026},
      eprint={2609.01442},
      archivePrefix={arXiv},
      primaryClass={math.DG},
      url={https://arxiv.org/abs/2609.01442}, 
}
\end{document}